\documentclass{statsoc}
\usepackage{amsfonts,amssymb,amsmath,amscd,latexsym,color}
\usepackage{natbib}
\bibpunct{(}{)}{,}{a}{,}{,}

\usepackage{graphicx}
\usepackage{fullpage}
\newtheorem{theorem}{Theorem}
\newtheorem{example}{Example}
\newtheorem{lemma}{Lemma}
\newtheorem{corollary}{Corollary}
\newenvironment{ex}
  {\noindent\begin{example}\begin{em}}
  {\end{em}\hfill$\blacktriangleleft$\end{example}}

\def \be{\begin{equation}}
\def \ee{\end{equation}} 
\def \dist{\rightsquigarrow} %{\stackrel{n\rightarrow\infty}{\leadsto}}
\newcommand{\by}{\mathbf{y}}
\newcommand{\bz}{\mathbf{z}}

\newcommand{\vs}{\vspace{0.5cm}}

\newcommand{\E}{\mathbb{E}}

\renewcommand{\P}{\mathbb{P}}
\newcommand{\R}{\mathbb{R}}
\newcommand{\Z}{\mathbb{Z}}

\newcommand{\btheta}{\boldsymbol{\theta}}
\newcommand{\feta}{\boldsymbol{T}}

\def\1{1\!{\rm l}}  
\newcommand{\model}{\mathfrak{M}}
\title[Relevant statistics for Bayesian model choice]{Relevant statistics for Bayesian model choice}

\author{Jean-Michel Marin}
\address{I3M, UMR CNRS 5149, Universit\'e Montpellier 2, France.}
\email{jean-michel.marin@math.univ-montp2.fr}
\author{Natesh S.~Pillai}
\address{Department of Statistics, Harvard University, Cambridge, USA.}
\email{pillai@fas.harvard.edu}
\author{Christian P.~Robert}
\address{Universit\'e Paris-Dauphine, University of Warwick, and CREST, Paris, France.}
\email{xian@ceremade.dauphine.fr,robert@ensae.fr}
\author[Rousseau {\it et al.}]{Judith Rousseau}
\address{Universit\'e Paris Dauphine, CEREMADE, ENSAE and CREST, Paris, France.}
\email{rousseau@ensae.fr}

\begin{document}

\maketitle

\begin{abstract} 
The choice of the summary statistics used in Bayesian inference and in particular in ABC algorithms has
bearings on the validation of the resulting inference. Those statistics are nonetheless customarily used in
ABC algorithms without consistency checks. We derive necessary and sufficient conditions on summary statistics 
for the corresponding Bayes factor to be convergent, namely to asymptotically select the true model. Those 
conditions, which amount to the expectations of the summary statistics differing asymptotically under the two 
models, are quite natural and can be exploited in ABC settings to infer whether or not a choice of summary statistics is 
appropriate, via a Monte Carlo validation.
\end{abstract}


\section{Introduction}
\subsection{Summary statistics}

In \cite{robert:cornuet:marin:pillai:2011}, the authors showed that the now popular ABC (approximate Bayesian
computation) method \citep{tavare:balding:griffith:donnelly:1997,pritchard:seielstad:perez:feldman:1999,
toni:etal:2009} is not necessarily validated when applied to Bayesian model
choice problems, in the sense that the resulting Bayes factors may fail to pick the correct model even
asymptotically. 

The ABC algorithm is progressively getting accepted as a necessary component of
the Bayesian toolbox for handling intractable likelihoods. Since it is not the
central topic of this article, but rather both a motivation and an immediate
application domain for our derivation, we do not embark upon a complete
description of its implementation, referring to
\cite{marin:pudlo:robert:ryder:2011} and \cite{fearnhead:prangle:2012} for
details.  We simply recall here that the core feature of this approximation
technique is to run simulations $(\btheta,\bz)$ from the prior distribution and
the corresponding sampling distribution until a statistic $\feta(\bz)$ of the
simulated pseudo-data $\bz$ is close enough to the corresponding value of the
statistic $\feta(\by)$ at the observed data $\by$. The degree of proximity
(also called the tolerance) can be improved by an increase in the computational
power. However the choice of the statistic $\feta$ is particularly crucial in
that the resulting (approximately Bayesian) inference relies on this statistic
and only on this statistic. It thus impacts the resulting inference much more
than the choices of the tolerance distance and of the tolerance value.

When conducting ABC model choice
\citep{grelaud:marin:robert:rodolphe:tally:2009}, the outcome of the
ideal algorithm associated with zero tolerance and zero Monte Carlo error
is the Bayes factor
$$
B^{\feta}_{12}(\by)=\dfrac{\int \pi_1(\btheta_1)g_1^{\feta}(\feta(\by)|\btheta_1)\,\text{d}\btheta_1}
                          {\int \pi_2(\btheta_2)g_2^{\feta}(\feta(\by)|\btheta_2)\,\text{d}\btheta_2}\,,
$$
namely the Bayes factor for testing $\model_1$ versus $\model_2$ based on the sole observation
of $\feta(\by)$. This value most often differs from the Bayes factor $B_{12}(\by)$ based on the whole data
$\by$. As discussed in \cite{didelot:everitt:johansen:lawson:2011} and
\cite{robert:cornuet:marin:pillai:2011}, in the specific case when the statistic $\feta(\by)$ is sufficient for
both $\model_1$ and $\model_2$, the difference between both Bayes factors can be expressed as
\be\label{eq:summayes}
B_{12}(\by) = \dfrac{h_1(\by)}{h_2(\by)}\,B^{\feta}_{12}(\by)\,,
\ee
where the ratio of the $h_i(\by)$'s often behaves like a likelihood ratio of the same order as the data size $n$.  The
discrepancy revealed by the above is such that ABC model choice cannot be trusted without further checks.
Indeed, even in the limiting ideal case, i.e.~when the ABC algorithm 
achieves a zero tolerance, the ABC odds ratio does not take into account the features of the data besides the
value of $\feta(\by)$.  \cite{robert:cornuet:marin:pillai:2011} warn that this difference can be such that
$B^{\feta}_{12}(\by)$ leads to an inconsistent model choice. (The same is obviously true for point estimation,
e.g.~when considering the extreme case of an ancillary $\feta(\by)$.) This is also the reason
why \cite{ratmann:andrieu:wiujf:richardson:2009,ratmann:andrieu:wiujf:richardson:2010} consider the alternative
approach of assessing each model on its own under several divergence measures, defining a new algorithm 
they denote $\text{ABC}_\mu$.

Beyond ABC applications, note that many fields report summary statistics in
their publications rather than the raw data, for various reasons ranging from
confidentiality to storage, to proprietary issues. For instance, a dataset may
be replaced by several $p$-values, $p_i(\by)$, against several specific
hypotheses. Handling a model choice problem based solely on
$\feta(\by)=(p_1(\by),\ldots,p_k(\by))$ is therefore a relevant issue, with the
coherence of the corresponding Bayes factor at stake.

Another relevant instance outside the ABC domain is provided in
\cite{dickey:gunel:1978}, who exhibit the above differences in the Bayes
factors when using a non-sufficient statistic, including an example where the
limiting Bayes factor, as the sample size grows to infinity, is $0$ or $\infty$.
Similarly, \cite{walsh:raftery:2005} compare point processes via Bayes factors
constructed on summary statistics. They discuss those summary statistics (second
order statistics and some based on Vorono{\"\i} tesselations) depending on the
misclassification rates of the corresponding Bayes factors through a simulation
study. However, the connection with the genuine Bayes factor is not pursued. (A
connection with the ABC setting appears in the conclusion of the paper, though,
with a reference to \cite{diggle:graton:1984} which is often credited as one
originator of the method.)

The purpose of the current paper is to study asymptotic conditions on the
statistic $\feta$ under which the Bayes factor $B^{\feta}_{12}(\by)$ either
converges to the correct answer or it does not. We obtain a precise
characterisation of consistency in terms of the limiting distributions of the
summary statistic $\feta(\by)$ under both models, namely that the true
asymptotic mean of the summary statistic $\feta(\by)$ cannot be recovered under
the wrong model, except for nested models. As explained below, this
characterisation shows that using point estimation statistics as summary
statistics is rarely pertinent for testing while ancillary statistics are more
likely candidates, at least formally. Once stated, the condition on the
statistic $\feta$ is quite natural in that the Bayes factor will otherwise
favour the simplest model. Our main result implies that a validation of summary
statistics providing convergent model choice is available for ABC algorithms.
The practical side is computational in that the mean values of the summary
statistics can be checked by simulation. Further properties of the vector of
summary statistics can also be tested via these simulations, including the
comparison of several summary statistics or, equivalently, the selection of the
most discriminant components of the above vector.

\subsection{Insufficient statistics}
\label{subsec:insu}

The above connection \eqref{eq:summayes} between the Bayes factor based on the whole data $\by$ and
the Bayes factor based on the summary $\feta(\by)$ is only valid when the
latter is sufficient for both models. In this setting, and only in this
setting, the ratio of the $h_i$'s in \eqref{eq:summayes} is equal to one solely
when the statistic $\feta$ is furthermore sufficient across models $\model_1$
and $\model_2$, i.e.~for the collection $(m,\btheta_m)$ of the model index and
of the parameter. A rather special instance where this occurs is the case of
Gibbs random fields \citep{grelaud:marin:robert:rodolphe:tally:2009}.
Otherwise, the conclusion drawn from using $\feta(\by)$ necessarily differs from the
conclusion drawn from using $\by$. The same is obviously true outside the sufficient
case, which implies that the selection of a summary statistic must be evaluated
against its performances for model choice, because it is not guaranteed per se.
The following example illustrates this point:

\begin{ex}\label{laplace-vs-gaussian} To illustrate the impact of the choice of a summary statistic on the Bayes factor, we
consider the comparison of model $\model_1$: $\by\sim\mathcal{N}(\theta_1,1)$ with model $\model_2$:
$\by\sim\mathcal{L}(\theta_2,1/\sqrt{2})$, the Laplace or double exponential distribution with mean $\theta_2$
and scale parameter $1/\sqrt{2}$, which has a variance equal to one. Since it is irrelevant for consistency
issues, we assume throughout the paper that the prior probabilities of both models $\model_1$ and $\model_2$
are equal to $1/2$.

In this formal setting, we considered the following statistics:
\begin{enumerate}
\item[--] the sample mean $\overline \by$;
\item[--] the sample median $\text{med}(\by)$;
\item[--] the sample variance $\text{var}(\by)$;
\item[--] the median absolute deviation $\text{mad}(\by)=\text{med}(|\by-\text{med}(\by)|)$;
\item[--] the sample fourth moment $n^{-1} \sum_{i=1}^n y_i^4$;
\item[--] the sample sixth moment $n^{-1} \sum_{i=1}^n y_i^6$.
\end{enumerate}
Given the models under comparison, the first statistic is sufficient only for the Gaussian model,
the second, fifth and sixth statistics are not sufficient but their distributions depend on $\theta_i$ in both models, 
while both the sample variance and the median absolute deviation are ancillary statistics. 

As explained later in Section \ref{sec:consequence}, the most important feature
of those statistics is that all statistics but the fourth one 
have the same expectation under both models (when using appropriate values of
the $\theta_i$'s) while the median absolute deviation always has a different
expectation under model $1$ and model $2$.

Since we are facing standard models in this artificial example, the analytic
computation of the true Bayes factor would be possible, even in the Laplace
case. However, if we base our inference only on one or several of the above
statistics, the computation of the corresponding Bayes factors requires an ABC
step.  Fig.  \ref{fig:norla1} shows the distribution of the posterior
probability that the model is normal (as opposed to Laplace) when the data is
either normal or Laplace and when the summary statistic in the ABC algorithm is
the collection of the first three statistics above. The outcome is thus that
the estimated posterior probability has roughly the same predictive
distribution under both models, hence ABC based on those summary statistics is
not discriminative. Fig.  \ref{fig:norla2} represents the same outcome when the
summary statistic used in the ABC algorithm is only made of the median absolute
deviation of the sample. In this second case, the two distributions of the
estimated posterior probability are quite opposed under each model,
concentrating near zero and one as the number of observations $n$ increases,
respectively. Hence, this summary statistic is highly discriminant for the
comparison of the two models. From an ABC perspective, this means that using
the median absolute deviation is then satisfactory, as opposed to the first
three statistics. Finally, Fig. \ref{fig:4th} and \ref{fig:4+6th} represents
the same outcome when the summary statistics used in the ABC algorithm are
respectively the empirical fourth moment and both the empirical fourth and
sixth moments. When using solely the empirical fourth moment, the posterior
probability for the normal model is highly concentrated near 1 when the
observations are normally distributed, while the posterior probability for the
normal model slowly decreases to zero with the number of observations when they
are Laplace distributed. When using both the fourth and the sixth moments, the
convergence (to zero) in the Laplace case occurs faster. We note that the
distance used for the latter case is an Euclidean distance with weights $1$ and
$1/100$ on the fourth and the sixth components, in order to compensate for the
one-hundred-fold larger values of the square differences of the sixth moments.
Using a regular Euclidean distance led to account only for the empirical sixth
moment statistic.
In Section \ref{subsec:GaLap}, the experimental results
obtained in Fig. \ref{fig:4th} and \ref{fig:4+6th} will be analysed in terms of
theoretical results of \ref{sec:consequence}.  \end{ex}

\begin{figure}[ht]
 \centerline{\includegraphics[width=9truecm]{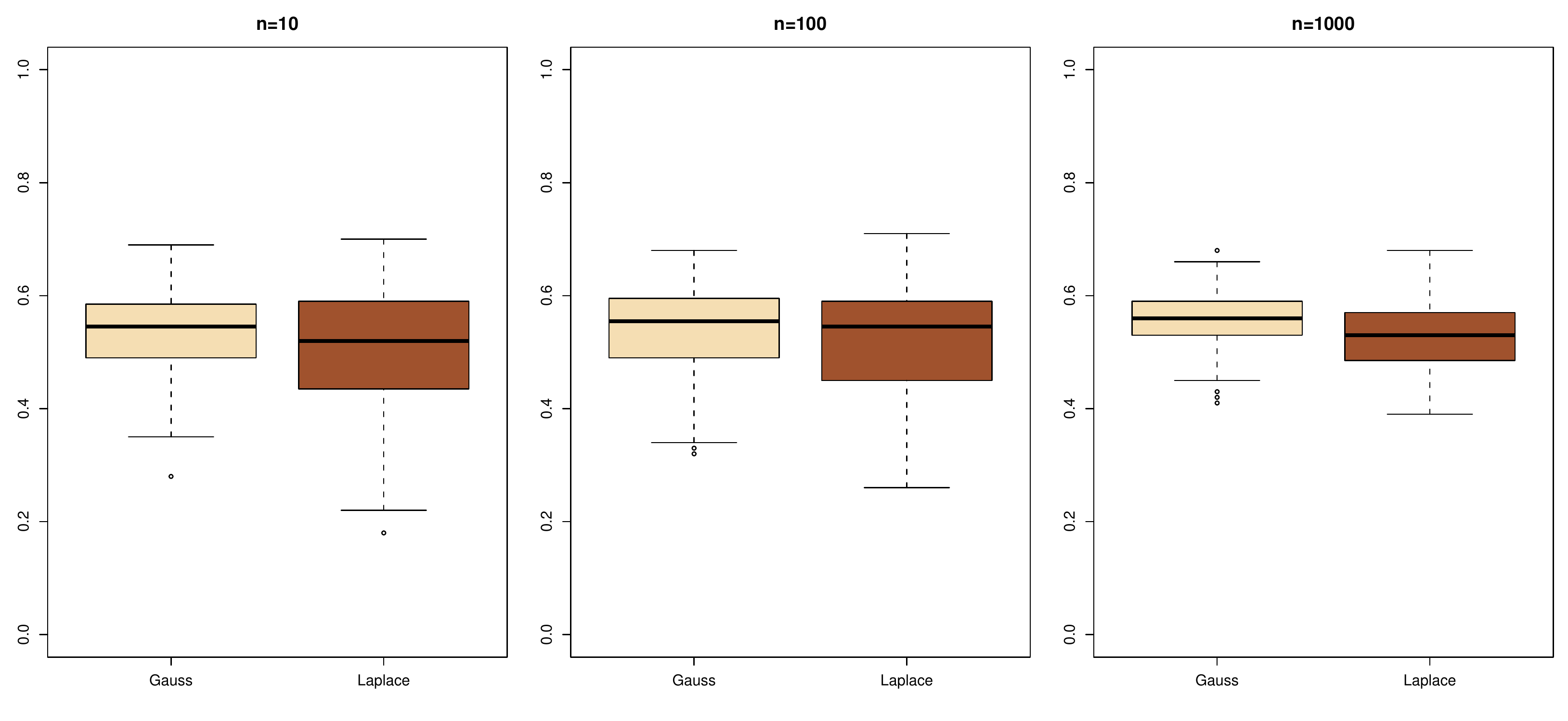}}
\caption{Comparison of the distributions of the posterior probabilities that the data is from a normal model
{\em (as opposed to a Laplace model)} with unknown mean $\theta$ when the data is made of $n=10,100,1000$ observations 
{\em (left, centre, right, resp.)} either from a Gaussian or
Laplace distribution with mean equal to zero and when the summary statistic in the ABC algorithm is 
the vector made of the collection of \textbf{the sample mean, median, and variance}. The ABC algorithm uses a
reference table of $10^4$ simulations ($5,000$ for each model) from the prior $\theta\sim \mathcal{N}(0,4)$
and selects the tolerance $\epsilon$ as the $1\%$ distance quantile over those simulations.}
\label{fig:norla1}
\end{figure}

\begin{figure}[ht]
 \centerline{\includegraphics[width=9truecm]{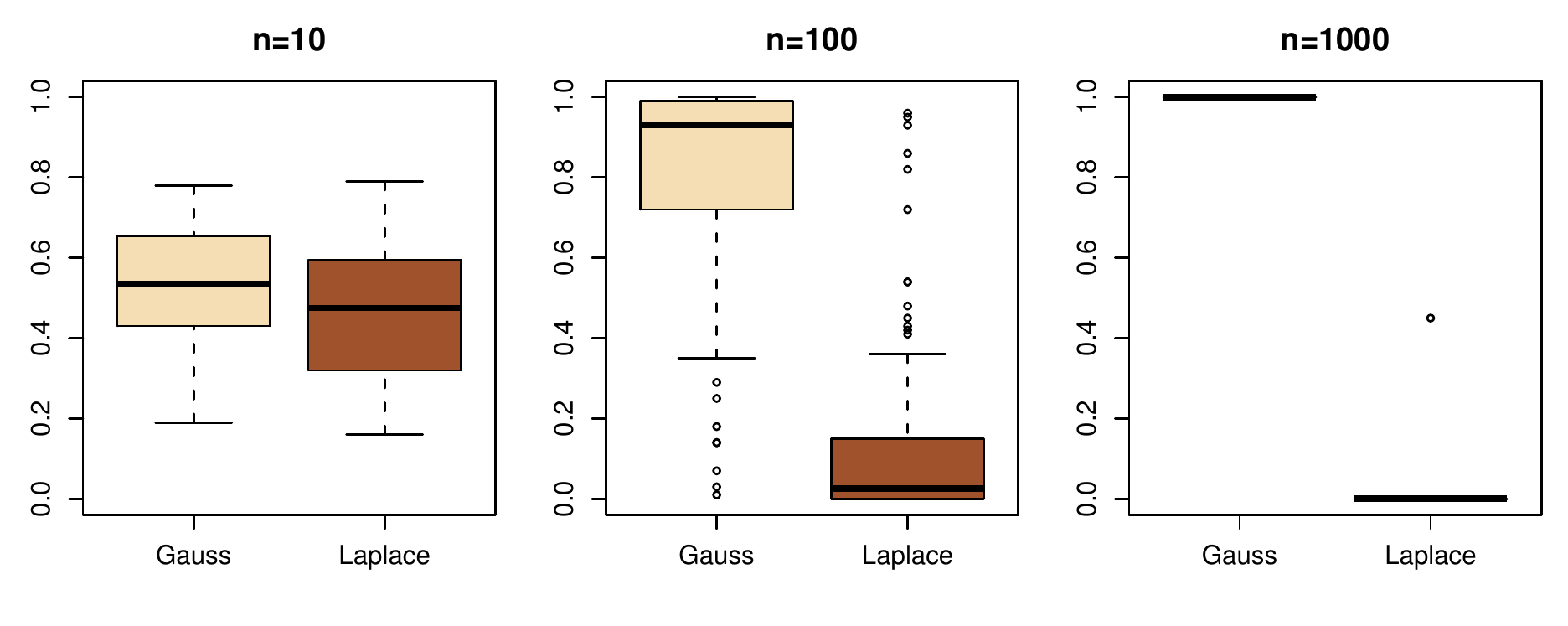}}
\caption{Same legend and same calibration as in Fig. \ref{fig:norla1} when the ABC algorithm is 
based on \textbf{the median absolute deviation} of the sample as its sole summary statistic.}
\label{fig:norla2}
\end{figure}

\begin{figure}[ht]
\centerline{\includegraphics[height=4truecm]{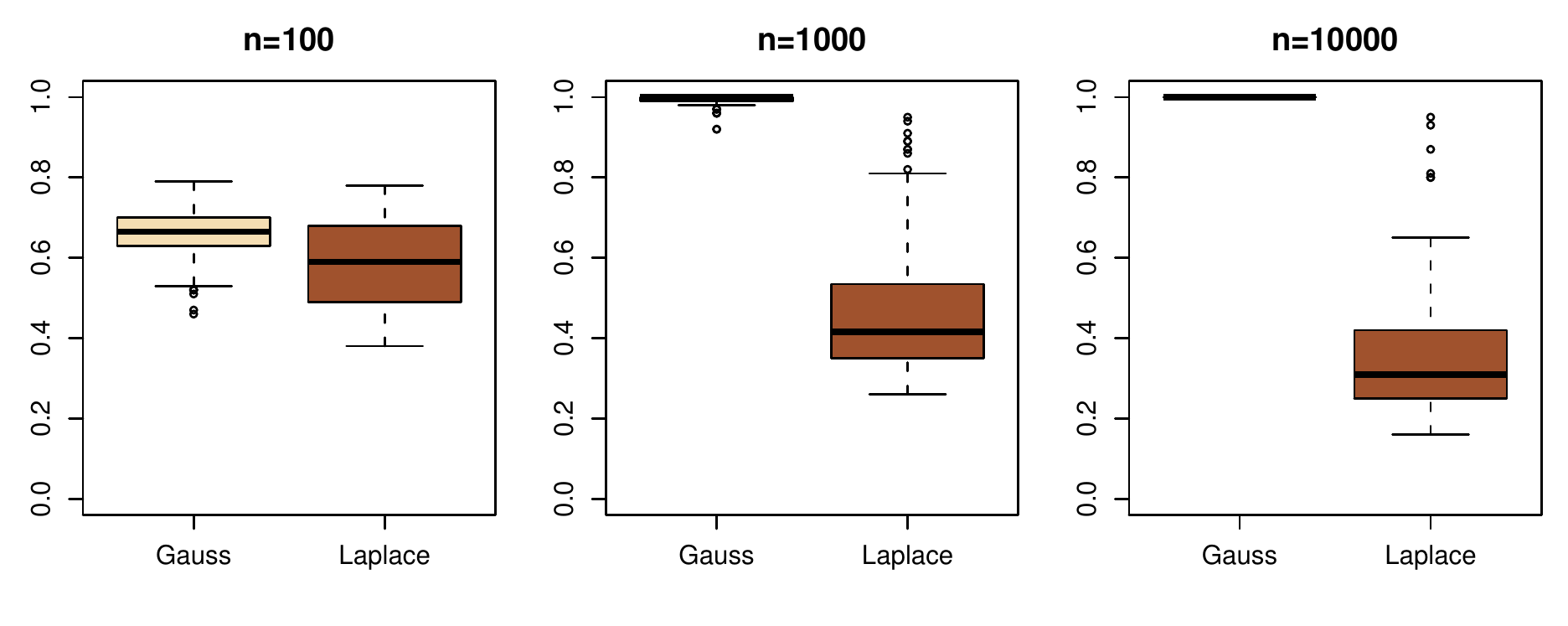}}
% \caption{Comparison of the distributions of the posterior probabilities that the data is from
% a Gaussian model {\em (as opposed to a Laplace model)} with unknown mean parameter when the data
% is made of $n=100,1000,10000$ observations generated from either from a Gaussian or Laplace
% distribution \textbf{with mean equal to zero} and when the summary statistic in the
% ABC algorithm is restricted to \textbf{the empirical fourth moment}. The
% ABC algorithm uses $10^4$ proposals ($5,000$ for each model) from the prior $\mathcal{N}(0,4)$
% and selects the tolerance as the 1\%~distance quantile.}
\caption{Same legend and same calibration as in Fig. \ref{fig:norla1},
for $n=100,1000,10000$ observations, when the ABC algorithm is 
based on \textbf{the empirical fourth moment} as its sole summary statistic.}
\label{fig:4th}
\end{figure}

\begin{figure}[ht]
\centerline{\includegraphics[height=4truecm]{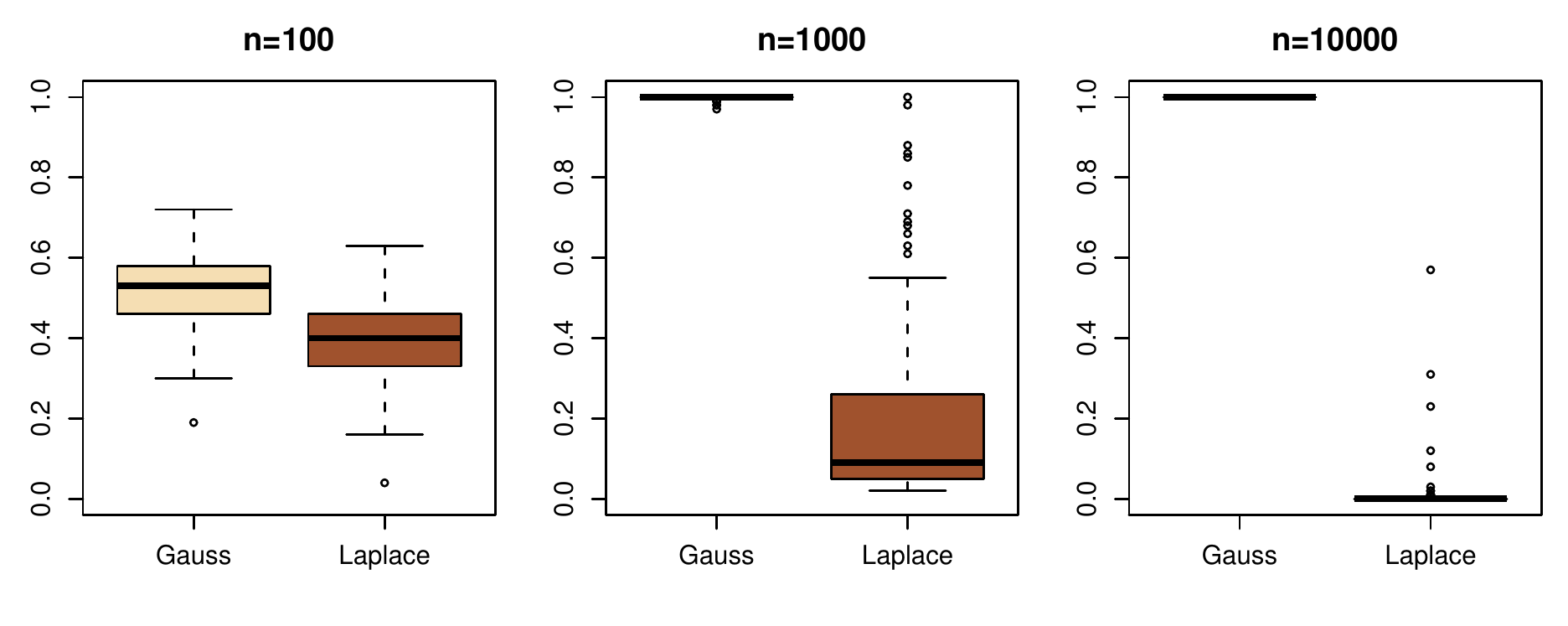}}
\caption{Same legend as in Fig. \ref{fig:4th} when the ABC algorithm is based on both
\textbf{the fourth and sixth empirical moments} as summary statistics.}
\label{fig:4+6th}
\end{figure}

The above example illustrates very clearly the major result of this paper, namely that the mean behaviour of
the summary statistic $\feta(\by)$ under both models under comparison is fundamental for the convergence of the
Bayes factor, i.e.~of the Bayesian model choice based on $\feta(\by)$. This result, described in the next
section, thus brings an answer to the question raised in \cite{robert:cornuet:marin:pillai:2011}
about the validation of ABC model choice, although it may require additional simulation experiments in
realistic situations.

The paper is organised as follows: Section \ref{sec:ZeSection} contains the
theoretical derivation of the asymptotic behaviour of the Bayes factor $B^{\feta}_{12}(\by)$,
Section \ref{sec:asmpt} covering our main assumptions and
exhibiting the asymptotic behaviour of the marginal likelihoods, Section
\ref{sec:consequence} detailing the consequences of this result for model
choice based on summary statistics.  Section \ref{sec:illustre} illustrates the
relevance of our criterion for evaluating summary statistics, including a non-trivial
population genetics example. Section \ref{sec:checkstat} details the practical implementation
of a validation mechanism based on the above results. Section \ref{seccon}
concludes the paper with a short discussion. 

\section{Convergence of Bayes factors using summary statistics}\label{sec:ZeSection}

\vs Let $\by = (y_1,\ldots,y_n)$ be the observed sample, not necessarily iid.  We denote by $\by \sim \P^n$ the
true distribution of the sample, and by $\feta(\by)=\feta^n=(T_1(\by), T_2(\by),\cdots,T_d(\by))$ a
$d$-dimensional vector of summary statistics, $\feta^n\sim G_n$. The distribution $G_n$ is the projection of
$\P^n$ under the map $\feta^n: \mathbb{R}^n \mapsto \mathbb{R}^d$ and we denote its density by $g_n$. 
% We denote the data by $X^n$ with $n$ denoting the sample size. \marginpar{\tcr{Its not clear if need i.i.d
% data. Try to see, if we really need this.}}

There are two competing models $\model_1$ and $\model_2$ that we wish to compare:
\begin{itemize}
\item[--] under $\model_1$, $\by\sim F_{1,n}(\cdot|\theta_1)$ where $\theta_1\in\Theta_1 \subset \R^{p_1}$,
\item[--] under $\model_2$, $\by\sim F_{2,n}(\cdot|\theta_2)$ where $\theta_2\in\Theta_2 \subset \R^{p_2}$.
\end{itemize}

The distributions of $\feta^n$ under $\model_1$ and $\model_2$ are denoted by $G_{1,n}(\cdot|\theta_1)$ and 
$G_{2,n}(\cdot|\theta_2)$, respectively. We also assume that the distribution functions $F_{i,n}(\cdot|\theta_i)$,
$G_{i,n}(\cdot|\theta_i)$ have densities $f_{i}(\cdot|\theta_i)$ and $g_{i}(\cdot|\theta_i)$ with respect
to some dominating measures $(i=1,2)$, respectively. 
Under the respective prior distributions $\pi_1$ and $\pi_2$ on
$\theta_1$ and $\theta_2$, the posterior distributions given $\feta^n$ are denoted by $\pi_1(\cdot|\feta^n)$ and $\pi_2(\cdot|\feta^n)$. 
%%%%%%%%%%%%

\subsection{Assumptions and asymptotic behaviour of the marginal likelihoods}\label{sec:asmpt}
%%%%%%%%%
Before stating the main result in the paper, we detail theoretical assumptions on both the models and the
summary statistics under which the main result holds. 

We start with a brief primer on our notations. The letter $C$ denotes a generic positive constant (independent
of $n$), whose value may change from one occurrence to the next, but is  of no consequence.  
We write $a\wedge b$ to denote $\min (a,b)$. For two sequences $\{a_n\},\{b_n\}$ of real numbers, 
$a_n \lesssim b_n$ (resp. $\gtrsim$) means $a_n \leq C b_n$ (resp. $a_n \geq Cb_n)$. Similarly,
$a_n \sim b_n$ means that 
$$
{1/C} \leq \liminf_{n \rightarrow \infty} |a_n/b_n|\leq \limsup_{n \rightarrow \infty} |a_n/b_n| \leq C \;.
$$   
The symbol $\dist$ denotes convergence in distribution. 

Technical assumptions that are necessary for establishing the main result of the paper are as follows:
\begin{enumerate}
\renewcommand{\theenumi}{{\bf A\arabic{enumi}}}

\item\label{A1} There exist a sequence of positive real numbers $\{v_{n}\}$ converging to $+\infty$,  
a distribution $Q$ on $\mathbb{R}^d$, 
% a symmetric, $d\times d$  positive definite, matrix  $V_0$, 
and a vector $\mu_0 \in \mathbb{R}^d$, such that
$$
% v_{n}  V_0^{-1/2}( \feta^n- \mu_0) \dist Q, 
v_{n}  (\feta^n-\mu_0) \dist Q, 
\quad \mbox{ under } G_n  \;.
$$
%and for all $\epsilon, M>0$, there exists a subset   $E_n$ such that $G_n(E_n) > 1 - \epsilon$ and uniformly on $\{t \in E_n; v_n|t - \mu_0| < M \}$,
%$$ q\big\{v_n V_0^{-1/2}(t - \mu_0) \big\}\lesssim v_n^{-d} g_{n}(t) \lesssim q\big\{v_n V_0^{-1/2}(t - \mu_0)\big\}.$$

%\item \label{A2} For every $\theta_i \in \Theta_i$, $i\in\{1,2\}$, there exist $d \times d$  symmetric 
%positive definite matrices $V_i(\theta_i)$ and vectors $\mu_i(\theta_i) \in \mathbb{R}^d$ such that  
%$$
%v_{n}  V_i(\theta_i)^{-1/2}( \feta^n- \mu_i(\theta_i)) =O_P(1), 
%\quad \mbox{ under } G_{i,n}(\cdot|{\theta_i}) \;.
%$$

\item \label{A3} For $i\in \{ 1,2\}$, there exist sieves $\mathcal F_{n,i} \subset \Theta_i$  
and constants $\tau_i , \alpha_i, C_i, x_{0,i}>0 $, such that 
\begin{equation}\label{mass:Fni}
\pi_i(\mathcal F_{n,i}^c ) = o( v_n^{-\tau_i})\,.
\end{equation}
For all $\theta_i \in \mathcal F_{n,i}$, the asymptotic means $\mu_i(\theta_i) \in \mathbb{R}^d$ 
of $\feta^n$ under this model satisfy: for all  $0< x < x_{0,i}v_n$
\begin{equation}\label{tests}
G_{i,n} \Big[ v_n |\feta^n - \mu_i(\theta_i)| > x \,\Big| \theta_i \Big] \leq C_i x^{-\alpha_i}\,.
\end{equation}
We define the sets $S_{n,i} \subset \mathcal{F}_{n,i}$ $(i = 1,2)$ as
$$
S_{n,i}(u) = \big \{ \theta_i \in \mathcal F_{n,i}; |\mu_i(\theta_i) -\mu_0| \leq u \,v_n^{-1} \big\}, \quad u > 0 \;.
$$
\end{enumerate}

\medskip
We say that $\model_i$ is {\em compatible with} $\feta^n$ if 
$$
\inf\{ |\mu_i(\theta_i) - \mu_0 | ; \theta_i \in \Theta_i\} = 0\,,
$$
meaning that the asymptotic mean of $\feta^n$ is found within the range of the means of $\feta^n$ in model
$\model_i$.

\begin{enumerate}
\renewcommand{\theenumi}{{\bf A\arabic{enumi}}}
\setcounter{enumi}{2}

\item\label{A4}   
%Define the sets $S_{n,i} \subset \mathcal{F}_{n,i}$ $(i\in \{ 1,2\})$ as
%$$
%S_{n,i}(u) = \big \{ \theta_i \in \mathcal F_{n,i}; |\mu_i(\theta_i) -\mu_0| \leq u \,v_n^{-1} \big\}, \quad u > 0 \;.
%$$
If $\model_i$ is compatible with $\feta^n$,
%$\inf\{ |\mu_i(\theta_i) - \mu_0 | ; \theta_i \in \Theta_i\} = 0$,
%then there exists a constant $d_i < \tau_i$ with $d_i < \alpha_i - 1$ such that
then there exists a constant 
$
d_i < \tau_i \wedge (\alpha_i - 2)
$
such that
\begin{eqnarray} \label{eqn:A5e2}
\pi_i( S_{n,i}(u) ) \sim u^{d_i}v_n^{-d_i} , \quad \forall u \lesssim v_n \;.
%, \quad \mbox{ for some } c>0
\end{eqnarray}
%where $S_{n,i}(u)$, $\tau_i$ and $\alpha_i$ are defined in assumption (\ref{A3}).

\item\label{A5} If 
%$\inf\{ |\mu_i(\theta_i) - \mu_0 | ; \theta_i \in \Theta_i\} = 0$,  
$\model_i$ is compatible with $\feta^n$,
then for any $\epsilon > 0$ there exist $U, \delta >0$  and a set 
$ E_n$ such that for all $\theta_i \in S_{n,i}(U)$
\begin{equation} \label{eqn:A6eq1}
E_n \subseteq \{ t; g_{i}(t|\theta_i)  \geq \delta  g_n(t) \} 
\quad \text{and} \quad G_n\left( E_n^c \right) < \epsilon\,.
\end{equation}
\end{enumerate}

Even though these assumptions might appear overwhelming, we claim that
(\ref{A1})---(\ref{A5}) are both mild and relatively easy to check in applications.
A detailed discussion on those assumptions is provided in Section
\ref{subsec:ass}. Furthermore, we will later illustrate why they hold in 
both the Gaussian versus Laplace example (Section \ref{subsec:GaLap}) and a
realistic population example (Section \ref{sec:popgenX}). 

The following result provides a fundamental control on the convergence rate of the marginal likelihoods.  
In Lemma \ref{thm:normasympt}, $m_{1}(\cdot)$ and $m_{2}(\cdot)$ denote the
marginal densities of $\feta^n$ under models $\model_1$ and $\model_2$, respectively, namely $(i=1,2)$
\begin{equation}\label{eqn:ML}
m_i(t) = \int_{\Theta_i} g_i(t|\theta_i)\,\pi_i(\theta_i)\,\text{d}\theta_i\,.
\end{equation}

\begin{lemma}\label{thm:normasympt}
Under assumptions (\ref{A1})--(\ref{A5}), for $i=1,2$, there exist constants $C_l, C_u = O_{\P^n}(1)$ such that, if 
$\model_i$ is compatible with $\feta^n$,
%$\inf\{ |\mu_i(\theta_i) - \mu_0 | ; \theta_i \in \Theta_i \} = 0$ 
and $ \tau_i > d_i$, $\alpha_i> d_i+2$,
\begin{equation} \label{th1:1}
C_l v_n^{-d_i}  \leq \frac{m_i(\feta^n) }{ g_n(\feta^n) } \leq C_u v_n^{ - d_i} 
\end{equation}
and, if $\model_i$ is not compatible with $\feta^n$,
%$\inf\{ |\mu_i(\theta_i) - \mu_0 | ; \theta_i \in \Theta_i \} > 0$,
\begin{equation} \label{th1:2}
\frac{m_i(\feta^n) }{ g_n(\feta^n) }  = O_{\P^n} [v_n^{- \tau_i  } + v_n^{ - \alpha_i} ]. 
\end{equation}
\end{lemma}

The above lemma, or more precisely \eqref{th1:1}, provides an equivalence result for the marginal densities
$m_i(\feta^n)$ when $\mu_0 \in \{ \mu_i(\theta_i), \theta_i \in \Theta_i\}$ but it does not specifically
require that $G_n$ belongs to model $\model_i$. Appendix 2 details the proof of Lemma \ref{thm:normasympt}. The
following result is a corollary on the use of $\feta^n$ for inference purposes other than model choice:

%\tc{This corollary also opens up as a new method for statistical inference; if we are interested in parameter
%estimation, forget the whole data but just focus on a good summary statistic. This has nothing to do with
%ABC..I will write expand on these thoughts in the final polishing stages}
\begin{corollary} \label{corollary:1} 
Under the assumptions of Lemma \ref{thm:normasympt}, if  $\model_i$ is compatible with $\feta^n$,
%$\mu_0 \in \{ \mu_i(\theta_i); \theta_i \in \Theta_i\} $, 
the posterior distribution $\pi_i(.| \feta^n)$ concentrates at the
rate $1/v_n$ on $\{ \theta_i; \mu_i(\theta_i)= \mu_0\}$, provided $ \tau_i > d_i$ and $\alpha_i> d_i+2$. Hence,
under the posterior distribution $\pi_i(.|\feta^n) $, $\mu_i(\theta_i)$ converges to $\mu_0$ at the rate
$1/v_n$.  \end{corollary}

\begin{proof}
Equation \eqref{th1:1} of Lemma \ref{thm:normasympt} yields  that $$\frac{m_i(\feta^n)}{g_n(\feta^n)} \gtrsim v_n^{-d_i} \;$$
with large probability.
For all sequences $\{w_n\}_n$ converging to $+\infty $, calculations performed in the proof of Lemma \ref{thm:normasympt} 
(see Appendix 2) yield that with probability going to 1 under $G_n$,
$$
\int_{S_{n,i} (w_n)^c}\frac{ g_i(\feta^n|\theta_i)}{g_n(\feta^n) } \, 
\pi_i(\theta_i)\, \text{d}\theta_i \lesssim w_n^{-\alpha_i} v_n^{- \alpha_i}+v_n^{  - \tau_i } = o(v_n^{-d_i})\;.
$$
Therefore the  posterior distribution of $\mu_i(\theta_i)$ has its tail probability given by
$$
\pi_i(|\mu_0 - \mu_i(\theta_i)|> w_n v_n^{-1} | \feta^n )  =\frac{\int_{S_{n,i} (w_n)^c}
g_i(\feta^n|\theta_i) \pi_i(\theta_i)\,\text{d}\theta_i}{ m_i(\feta^n) } = o_{\P^n}(1)
$$
and the corollary follows. \qed \end{proof}

\medskip
Lemma \ref{thm:normasympt} helps in understanding  the meaning of the parameter $d_i$ in assumption (\ref{A4})
when $\model_i$ is compatible with $\feta^n$.
%when $\inf\{ |\mu_i(\theta_i) - \mu_0|; \theta_i \in \Theta_i\} = 0$. 
Indeed, we then have
$$
\frac{m_i(\feta^n)}{g_n(\feta^n)} \sim v_n^{-d_i}\,,
$$
thus $\log (m_i(\feta^n)/g_n(\feta^n)) \sim -d_i \log v_n$ and
$v_n^{-d_i}$ appears as a penalisation factor resulting from integrating $\theta_i$ out 
%in model $\model_i$ 
in the very same spirit as the effective number of parameters appears in the DIC \citep{spiegbestcarl}
criterion and in the discussions in \citet{rousseau:07} and \citet{mengersen:rousseau:2011}. In regular models, $d_i$
corresponds to the dimension of $\mu(\Theta_i)$, leading to the usual BIC approximation; however, in
non-regular models, which may occur with the kind of applications where ABC methods are required, $d_i$ can be
different. This is illustrated in the examples of Section \ref{sec:illustre}.  We now present the major
implication of these results on the relevance of some summary statistics to compute Bayes factors.  

\subsection{Bayes factor consistency} \label{sec:consequence}

Lemma \ref{thm:normasympt} implies that the asymptotic behaviour of the Bayes factor is driven by the
asymptotic mean value of $\feta^n$ under both models. It is usual to assume that one of the competing models
is true, when studying the behaviour of testing procedures (here posterior probabilities and Bayes factors).
Here, in full generality, it is actually enough that one of the models is compatible with the statistic
$\feta^n$.  Hence, without loss of generality we assume that the true distribution belongs to model $\model_1$
and we first consider the case where  model $\model_2$ is {\em also compatible} with $\feta^n$, i.e.
$$
\inf\{ |\mu_0 - \mu_2(\theta_2) | ; \theta_2 \in \Theta_2 \} = 0\,.
$$ 
Under assumptions (\ref{A1})--(\ref{A5}),
$$
C_l v_n^{-(d_1-d_2)} \leq \dfrac{m_{1}(\feta^n) }{m_{2}(\feta^n)}  \leq C_u v_n^{-(d_1-d_2)},
$$
where $C_l, C_u = O_{\P^n}(1)$, irrespective of the true model. Thus the asymptotic behaviour of the Bayes
factor depends solely on the difference $d_1-d_2$. For instance, if $d_1>d_2$ (as in the embedded case) and
$G_n$ is in $\model_1$, the Bayes factor goes to 0, instead of infinity. If instead $d_1 = d_2$, the Bayes factor is
bounded from below and from above and is thus useless to separate the two models.  Note that the asymptotic
(non-convergent) behaviour remains the same even when $G_n$ is in neither model, provided   
$$
\inf\{ |\mu_0 - \mu_2(\theta_2) | ; \theta_2 \in \Theta_2 \} =  \inf\{ |\mu_0 - \mu_1(\theta_1) | 
; \theta_1 \in \Theta_1 \} = 0 \;.
$$

On the contrary, assume that the true distribution is in model $\model_1$ and that model $\model_2$ is not
compatible with $\feta^n$,
%$$
%\inf\{ |\mu_0 - \mu_2(\theta_2) | ; \theta_2 \in \Theta_2 \} >0\,,
%$$
then the Bayes factor, under assumptions (\ref{A1})--(\ref{A5}), satisfies
$$
  \dfrac{m_1(\feta^n) }{m_2(\feta^n)}  \geq C_\ell \min \left( v_n^{-(d_1- \alpha_2)} , v_n^{-(d_1-  \tau_2)}\right),
$$
and if $\min( \alpha_2, \tau_2) >d_1 $, 
 $$
 \lim_{n \rightarrow +\infty } \dfrac{m_1(\feta^n) }{m_2(\feta^n)} = +\infty,
$$
which leads to choosing the right model asymptotically. 
The above then implies the following consistency result, which is the core derivation of our paper, providing a
characterisation of relevant summary statistics:

\begin{theorem}  \label{thm:BFcon}
If, under assumptions (\ref{A1})--(\ref{A5}),  models $\model_1$ and $\model_2$ are both compatible with $\feta^n$,
and $\tau_i \wedge \alpha_i-2 > d_i$,
then the Bayes factor $B_{12}^{\feta^n}$ has the same asymptotic behaviour as $v_n^{-(d_1 - d_2)}$, irrespective
of the true model. Therefore, it always asymptotically selects the model 
with the smallest effective dimension $d_i$. 

If model $\model_1$ is compatible with $\feta^n$ and model $\model_{2}$ is incompatible with $\feta^n$, 
then
\begin{equation*}
%\begin{split}
%0=  \min_{i=1,2} \left( \inf\{ |\mu_0 - \mu_i(\theta_i) | ; \theta_i \in \Theta_i \}\right)  < \max_{i=1,2} \left( \inf \{ |\mu_0 - \mu_i(\theta_i) | ; \theta_i \in \Theta_i \}\right),
0=  \inf\{ |\mu_0 - \mu_1(\theta_1) | ; \theta_1 \in \Theta_1 \}  
<   \inf \{ |\mu_0 - \mu_{2}(\theta_{2}) | ; \theta_{2} \in \Theta_{2} \}\,,
%\end{split}
 \end{equation*}
and if $\min( \alpha_{2}, \tau_{2}) >d_1 $,
then the Bayes factor  $B_{12}^{\feta}$  is consistent.
\end{theorem} 

Note that, for ancillary statistics, the condition $\min( \alpha_{2}, \tau_{2})
>d_1 $ is vacuous since $\tau_2=\infty$ and $d_1=0$. The theorem therefore also
applies to compatible ancillary statistics.

An essential practical consequence of Theorem \ref{thm:BFcon} is that the Bayes
factor is merely driven by the means $\mu_i(\theta_i)$ and the relative
position of $\mu_0$ in both sets $\{ \mu_i(\theta_i); \theta_i \in \Theta_i\}
$, $i=1,2$. If $G_n$ is in neither model but $\mu_0$ belongs to $\{
\mu_1(\theta_1) , \theta_1 \in \Theta_1\} $ and not to $ \{ \mu_2(\theta_2) ,
\theta_2 \in \Theta_2\} $, then the Bayes factor will asymptotically select
$\model_1$. Note that the result does not cover the behaviour of the Bayes
factor when neither model is compatible with $\feta^n$, since there is no
simple characterisation in this case.

The following heuristic argument sheds some light on why the above results hold. 

Suppose the summary statistics (appropriately rescaled) are asymptotically
normal under each model. Assume that the Kullback-Leibler divergence
between the distributions of $\sqrt{n} (\feta^n - \mu_i(\theta_i))$ can be
approximated by the Kullback-Leibler divergence between the respective asymptotic
Gaussian distributions  
$$
\sqrt{n}|V_0|^{-1/2}q_\mathcal{G}\{ \sqrt{n} V_0^{-1/2}( \feta^n - \mu_0)\}
$$
and 
$$
\sqrt{n}|V_i(\theta_i)|^{-1/2}q_\mathcal{G}\{\sqrt{n} V_i(\theta_i)^{-1/2}( \feta^n - \mu_i(\theta_i)\}\,,
$$
where $V_0$, $V_1(\theta_1)$, and $V_2(\theta_2)$ denote the asymptotic variances under the various models,
where $|V|$ denotes the determinant of the matrix $V$,
and where $q_\mathcal{G}$ is the pdf of the standard Gaussian distribution. Then
% $$ \frac{ 1}{ n} KL( g_n(\feta^n), g_{i}(\feta^n|\theta_i) ) \approx \frac{ 1 }{ n } KL (|V_0|^{-1/2}q(
% \sqrt{n} V_0^{-1/2}( \feta^n - \mu_0)), |V_i(\theta_i)|^{-1/2}q(\sqrt{n} V_i(\theta_i)^{-1/2}( \feta^n -
% \mu_i(\theta_i))) $$
% then 
\begin{equation}\label{KL:approx} 
\frac{ 1}{ n} KL\{ g_n(\feta^n), g_{i}(\feta^n|\theta_i) \} \approx \frac{ (\mu_0 - \mu_i(\theta_i) )^t 
V_i(\theta_i)^{-1} (\mu_0 - \mu_i(\theta_i) )}{2} + o(1)\,.
\end{equation}
In that case a usual Laplace argument would imply that 
\begin{equation*}
\frac{ m_i(\feta^n) }{ g_n(\feta^n) } \approx \int \exp\{-\frac{ n(\mu_0 - \mu_i(\theta_i) )^t 
V_i(\theta_i)^{-1} (\mu_0 - \mu_i(\theta_i) )}{2}\}\,\pi_i(\theta_i) d\theta_i .
\end{equation*}
So that the difference between $\mu_0$ and $\mu_i(\theta_i)$ is the key measure to evaluate the distance
between $g_n $ and $g_{i,n}(\cdot|\theta_i)$. The above argument is purely illustrative since requiring 
\eqref{KL:approx} is very strong and not realistic in most cases.

%We need to expand on ancillary stats:
Formally, an ideal statistics $\feta^n$ would be an ancillary
statistics for both models with different expectation under both models. Indeed, in this case,
the sets $\{ \mu_i(\theta_i) , \theta_i \in \Theta_i\}$ $(i=1,2)$ are singletons and they only have to differ
for the Bayes factor to be consistent. For instance, in Example 1, both the empirical variance and
the empirical mad statistic are ancillary. In the first case, the expectation is the {\em same} under both
distribution, which explains why the Bayes factor cannot discriminate between models 
(Fig.~\ref{fig:norla1}). In the second case, the expectations differ, hence a consistent Bayes factor 
as exhibited in Fig.~\ref{fig:norla2}. Concerning the assumptions (\ref{A1})--(\ref{A5}), some simplifications
occur under ancillarity: 
\begin{itemize}
\item[--] assumption (\ref{A3}) must hold for a single distribution and $\mathcal F_{n,i}= S_{n,i}(u) =\Theta_i$; 
\item[--] assumption (\ref{A4}) holds automatically since $d_i=0$;
\item[--] assumption (\ref{A5}) must also hold for the fixed distribution of $\feta^n$ under model $\model_i$ (and
obviously holds when $\model_i$ is the true model).
\end{itemize}
Unfortunately, it is very hard to extract useful ancillary statistics from complex models: while examples of
ancillary statistics abound, for instance rank statistics \citep{sidak:hajek:sen:1999}, they either do not apply
to non-iid settings or have identical means under different models. Example 1 is thus truly a toy example in
that it constitutes  the exception to this remark. When considering the population genetics models
of Section \ref{sec:illustre}, we cannot provide such solutions.

In the special case of $\model_1$ being a submodel of $\model_2$, and if the true distribution belongs to the
smaller model $\model_1$, any summary statistic satisfies  
$$
\mu_0 \in  \{\mu_1(\theta_1); \theta_1 \in \Theta_1\} \subset \{ \mu_2(\theta_2); \theta_2\in \Theta_2\},
$$
so that the Bayes factor is of order
$v_n^{-(d_1 -d_2)}$. If the summary statistic is informative merely on a parameter which is the same under both
models, \textit{i.e.}, if $d_1 = d_2$, then the Bayes factor is not consistent. Else, $d_1 < d_2$ and the Bayes
factor is consistent under $\model_1$. If the true distribution does not belong to $\model_1$, then the same
phenomenon as described above occurs and the Bayes factor is consistent only if $\mu_1 \neq \mu_2 = \mu_0$.  
This case will be illustrated for a quantile distribution in Section \ref{sec:kant}. 

\subsection{About the assumptions (\ref{A1})--(\ref{A5})} \label{subsec:ass}

\noindent 
Assumptions (\ref{A1})--(\ref{A5}) may appear too stringent or too abstract to be of any
practical relevance and we now discuss why they make perfect sense.

Assumption (\ref{A1}) is quite natural.  It is often the case that summary statistics $\feta^n$ are
chosen as empirical versions of quantities of interest (under second order moment conditions) and it is natural
to assume that they concentrate since they are chosen to be both low dimensional and informative on some
aspects of the model (even though the result also applies to ancillary statistics). For instance, when the
summary statistics are empirical means or empirical quantiles, (\ref{A1}) is satisfied with $v_n =
\sqrt{n}$ and the Gaussian distribution being the limiting $Q$ (a most common occurrence). However, if
$\feta^n$ is a distance (e.g., of the type induced by chi-square like statistics) then $Q$ will be the
chi-square distribution. We also note that (\ref{A1}) holds for some ancillary statistics, like those of Example
1.

%This convergence is only required under the true distribution $G_n$. Still,
Assumption (\ref{A3}) requires that under each model $\feta^n$
concentrates around the model asymptotic mean values $\mu_i(\theta_i)$ at rate $v_n$,
even though it is not necessary to have convergence in distribution. More
precisely, (\ref{A3}) controls the moderate deviations of the estimator
$\feta^n$ from the asymptotic mean $\mu(\theta)$ under each model.  For instance, when
$\feta^n$ is an empirical mean, i.e., $\feta^n = n^{-1} \sum_{i=1}^n h(y_i)$
for a given function $h$, Markov inequality implies that for every $\theta_i \in  \Theta_i$, 
{\small \begin{eqnarray}\label{emp:mean}
 G_{i,n} \left[ \sqrt{n} |\feta^n - \mu_i(\theta_i) | > u \Big| \theta_i \right] \leq \frac{ \mathbb{E}\Big[ 
\left| \sum_{i=1}^n \{h(y_i)  - \mu_i(\theta_i)\} \Big| \theta_i \right|^p\Big] }{ u^p \,n^{p/2} } \lesssim u^{-p} ,
\end{eqnarray}}
for large values of $p$ (typically, larger than $d_i+2$)
and under very weak assumptions (much weaker than being in an i.i.d.~setting). 

\noindent Assumption (\ref{A4}) describes the behaviour of the prior distribution of the mean of $\feta^n$ near the
true asymptotic value $\mu_0$. This assumption needs  only hold on a compatible model and
it is often found in the Bayesian asymptotic literature, see for instance condition (2.5) 
of Theorem 1 in \citet{gvdv:06}. Usually referred to as
{\em the prior mass condition}, it corresponds to the fact that if the prior vanishes in regions where the
likelihood is not too small (i.e., near $\mu_0$ in our case) then the marginal becomes very small. The
exponents $d_i$ can be viewed as effective dimensions of the parameter under the posterior distributions, as discussed after Corollary \ref{corollary:1}. 
If the maps $\theta_i \mapsto \mu_i(\theta_i) $ are 
locally invertible near $\mu_0$, under
the usual continuity conditions on the maps $\theta_i \mapsto |\mu_0 - \mu_i(\theta_i)|$,  for any $u >0$,
there exists a finite collection of points $\theta^*_{ij} \in \Theta_i$ such that the sets $S_{n,i}(u) $ can
be bounded both from above and from below by sets of the form 
\begin{equation} \label{eqn:covballs}
\bigcup_{j = 1}^J \{\theta_j: |\theta_j - \theta_{ij}^*| \lesssim u v_n^{-1} \}, \quad  J \in \mathbb{N} \;. 
\end{equation}
Thus if the prior density $\pi_i$ is bounded from above and below near the points $\theta_{ij}^*$,  we
immediately deduce that $\pi_i\{S_{n,i}(u) \}\sim u^dv_n^{-d}$ and $d_i = d$, verifying (\ref{A4}).  In most cases
we will have $d_i \leq d$, since assuming that $d_i>d$ would imply that the prior density of $\mu(\theta) $
explodes at $\mu_0$.

\noindent Assumption (\ref{A5}) states that, if there are $\theta_i$'s such that
$\mu_i(\theta_i) = \mu_0$, then uniformly in $\theta$ close to one of those
$\theta_i$'s, $g_i(t|\theta_i)$ is bounded from below by $\delta g_n(t)$ on a set
having large probability in terms of $G_n$. There are various instances under
which this assumption is satisfied. First, if $\model_i$ is the true model and $\feta^n$ is ancillary under this model, it
automatically holds since for all $\theta_i \in S_{n,i}(u)$ $g_{i}(\cdot |\theta_i) = g_n(\cdot)$. Secondly, if
$v_n(\feta^n - \mu_i)$ converges in distribution to $Q$ and if the densities
are close, then (\ref{A5}) is satisfied. This requires in particular that
$\feta^n$ has the same support under $G_n$ and $G_i$, but not necessarily
that $G_i$ or $G_n$ are continuous distributions.  Assumption (\ref{A5}) may
become difficult to check when the sets $S_{n,i}(u)$ are not compact, which is
typically the case when the sets $\{\theta_i; \mu_i(\theta_i)=\mu_0\}$ are not
compact.  The important point here is that, in such cases, the posterior
distribution $\pi_i(\cdot|\feta^n)$ is not informative on the whole parameter
$\theta_i$ (at least no further than the prior) but instead informative on a
fraction of it, summarised by $\mu_i(\theta_i)$. Re-parametrising  $\theta_i $ into $(\mu_i(\theta_i), \psi_i)$ where $\psi_i$ represents the part of $\theta_i$ which is not informed by the asymptotic distribution of $\feta^n$,  $\feta^n$ is asymptotically ancillary for $\psi_i$. In such a case, (\ref{A5}) 
will still hold in situations where the prior  distribution does not assign too much mass
near the tails, so that the sieves $\mathcal F_{n,i}$ can be chosen not too large
 to ensure that the distributions $G_{i,n}$ of $\feta^n$ hardly depend on $\psi_i$.

\section{Illustrations}\label{sec:illustre}

\subsection{Gaussian versus Laplace distributions} \label{subsec:GaLap}

Recall that in the setting of Example \ref{laplace-vs-gaussian}, we denote by
$\model_1$ the Gaussian model and by $\model_2$ the Laplace model. In each
model, the prior on the mean $\theta_i$ is a centred Gaussian distribution
with variance $4$ and in each case the data are simulated under $\theta_i = 0$.
For illustrating our main result on consistency, we consider the summary
statistics made of the empirical fourth moment, $\feta^n = n^{-1} \sum_{i=1}^n
y_i^4$, such that $\mu_1(\theta) = \theta^4 + 3 + 6 \theta^2$ and
$\mu_2(\theta) = \theta^4 + 6 + 6 \theta^2$.

We now endeavour to check that assumptions (\ref{A1})--(\ref{A5}) hold for that statistic. Given that this is an
empirical moment, (\ref{A1}) is trivially satisfied as a consequence of the Central Limit theorem, with
$v_n=\sqrt{n}$. 

For assumption (\ref{A3}), $\mu_i(\theta_i)$ is already defined above. 
For both models, we set $\mathcal F_{n,1}  = \mathcal F_{n,2} = \{ |\theta| 
\leq 2 \sqrt{\log n}\}=\mathcal F_n$ for (\ref{A3}) to hold, so that
$$
\pi_1(\mathcal F_{n,1}^c) =\pi_2(\mathcal F_{n,2}^c) = o(n^{-1}) 
$$
under a Gaussian prior on $\theta$ under both models, which implies $\tau_i=2$. The second part of 
(\ref{A3}) is verified using Markov inequality.  Indeed, 
\begin{equation*}
\begin{split}
G_{i,n} \left[ \left. \left| n^{-1} \sum_{j=1}^n (y_j^4 - \mu_i(\theta))\right| > x
\right| \theta\right]
  & \leq \frac{\mathbb{E}_i[ (Y^4 - \mu_i(\theta))^4|\theta] }{n^2 x^4}\,.
\end{split}
\end{equation*}
which implies $\alpha_i=4$.

Addressing (\ref{A4}), in model $\model_1$ if the mean is equal to zero 
then $\mu_0 = 3$ and, in model $\model_2$ if $\theta_2=0$, 
then $\mu_0 =6$.  Thus $S_{n,1}(u)  $ and $S_{n,2}(u)$ can be bounded 
from above and below by balls of the form
$$
|\theta| \leq C\,u^{1/2} n^{-1/4}\,,
$$  
so that $d_1=d_2=1/2$ in those cases. Note that, if the mean is different from zero, $\mu_0 >3$ 
in model $\model_1$ and $\mu_0 > 6$ in model $\model_2$. Then $S_{n,1}(u)$ and
$S_{n,2}(u)$ can be bounded from above and below by balls in the form 
$$
|\theta^2  - \theta_*^2|\leq C\,u\,n^{-1/2}, \quad |\theta_*|>0
$$ 
so that $d_1=d_2=1$ in those cases. \\
Addressing (\ref{A5}), since both distributions satisfy Cramer condition, the empirical fourth
moment allows for an Edgeworth expansion under both models, which can be made uniform in sets in
the form $\{ |\theta| \leq C n^{-1/4}\}$, see \citet[Theorem 19.1]{battacharya:rao}.
Hence, (\ref{A5}) is satisfied.

In conclusion, if the true distribution belongs to model $\model_2$ and the mean is equal to zero, $\mu_0
\in \{ \mu_i(\theta); \theta \in \R\} $ for both $i=1,2$ and we have  $d_1= 1$ and $d_2 = 1/2$. On the other
hand, if the true distribution belongs to model $\model_1$ then $d_1 = 1/2$ and 
$$
\inf\{ |\mu_0 - \mu_2(\theta_2) | ; \theta_2\in \R\}>0\,.
$$
Following from Theorem \ref{thm:BFcon}, the Bayes factor is then consistent but at the rate
$n^{-1/4}$ under model $\model_2$. This is to some extent an accidental result, merely due to the fact that, in
that very special case when the mean is equal to zero, $d_1>d_2$. Fig. \ref{fig:4th} presented in Section
\ref{subsec:insu} illustrates the above discussion.  Finally, note that, if the mean is different from zero,
then a similar argument leads to the lack of consistency of the Bayes factor, since then
$d_1=d_2=1$ and $\mu_0 \in \{\mu_i(\theta); \theta \in \R\}$, for both $i=1,2$.

\subsection{Quantile distributions}\label{sec:kant}

We now consider the example of a four-parameter quantile distribution, defined through its quantile function
$$
Q(p;A,B,g,k) = A + B \left( 1 + 0.8
\frac{1-\exp\{-gz(p)\}}{1+\exp\{-gz(p)\}}\right) \left (1+z(p)^2 \right )^k z(p)  
$$
where $z(p)$ is the $p$th standard normal quantile and the parameters $A, B, g$ and $k$ represent location,
scale, skewness and kurtosis, respectively \citep{haynes:macgillivray:mengersen:1997}. While the quantile function
is well-defined, and the distribution easy to simulate, there is no closed-form expression for the corresponding
density function, which makes the implementation of an MCMC algorithm quite delicate.
\cite{allingham:king:mengersen:2009} introduce a ABC procedure that uses the order statistics as summary statistics.
We consider here a model choice perspective.
 
In this experiment, we set $A=0$ and $B=1$. We then oppose two models:
\begin{itemize}
\item[--] model $\model_1$, in which $g=0$, with a single unknown parameter $\theta_1=k$ and a prior
$\theta_1\sim\mathcal{U}[-1/2,5]$. In the simulation process, when $\model_1$ is true, we choose $\theta_1=2$.
\item[--] model $\model_2$, with two unknown parameters $\theta_2=(g,k)$ and a prior
$\theta_2\sim\mathcal{U}[0,4]\otimes\mathcal{U}[-1/2,5]$.
In the simulation process, when $\model_2$ is true, we choose $\theta_{2,1}=1$ and $\theta_{2,2}=2$.
\end{itemize}
This obviously is a case of embedded models, since $\model_1$ is a sub-model of $\model_2$. As in the previous
experiments, we use an ABC procedure relying on $10^4$ proposals from the prior and a tolerance set at the
$1\%$ quantile of the $L_1$ distances between some empirical quantiles. In the comparison below, we first use
the empirical quantile of order $10\%$ as sole summary statistic.  Then, we consider the empirical quantiles of
order $10\%$ and $90\%$, and, at last, the empirical quantiles of order $10\%$, $40\%$, $60\%$ and $90\%$.  The
results are summarised in Fig.~\ref{fig:quantiles}.  They show complete agreement with Theorem 1.

When the summary statistic $\feta^n$ is restricted to the empirical quantile of order
$10\%$, the Bayes factor is not consistent. Indeed, in such a case,
we have
$$
\mu_1(\theta_1)= \left (1+z(0.1)^2 \right )^{\theta_1} z(0.1)\,,
$$
and
$$
\mu_2(\theta_2)=\left( 1 + 0.8\frac{1-\exp(-\theta_{2,1}z(0.1))}{1+\exp(-\theta_{2,1}z(0.1))}\right) 
\left (1+z(0.1)^2 \right )^{\theta_{2,2}} z(0.1)\,.
$$
When $\model_1$ is true with $k=2$, then $\mu_0\approx -8.95$ and
$\inf\{|\mu_0-\mu_1(\theta_1)|;\theta_1\in\Theta_1\}=\inf\{|\mu_0-\mu_2(\theta_2)|;\theta_2\in\Theta_2\}=0$.
Similarly, when $\model_2$ is true with $g=1$ and $k=2$, then $\mu_0\approx -4.90$ and
$$
\inf\{|\mu_0-\mu_1(\theta_1)|;\theta_1\in\Theta_1\}=\inf\{|\mu_0-\mu_2(\theta_2)|;\theta_2\in\Theta_2\}=0\,.
$$
Therefore, the Bayes factor has the same asymptotic behaviour as $n^{-(d_1-d_2)/2}$,
irrespective of the true model. We can prove that $d_1=d_2=1$ in this case, therefore 
that the Bayes factor is not consistent.

When the summary statistics $\feta^n$ is the vector made of the empirical quantiles of order
$10\%$ and $90\%$, the Bayes factor is consistent. Indeed, in such a case,
we have 
$$
\mu_1(\theta_1)= \left[\left (1+z(0.1)^2 \right )^{\theta_1} z(0.1),\left (1+z(0.9)^2 \right )^{\theta_1} z(0.9)\right]\,,
$$
and
$$
\mu_2(\theta_2)=\left[\left( 1 + 0.8
\frac{1-\exp(-\theta_{2,1}z(0.1))}{1+\exp(-\theta_{2,1}z(0.1))}\right) 
\left (1+z(0.1)^2 \right )^{\theta_{2,2}} z(0.1),\right.
$$
$$
\left.\left( 1 + 0.8 \frac{1-\exp(-\theta_{2,1}z(0.9))}{1+\exp(-\theta_{2,1}z(0.9))}\right) 
\left (1+z(0.9)^2 \right )^{\theta_{2,2}} z(0.9)\right]\,.
$$
When $\model_1$ is true with $k=2$, then $\mu_0\approx [-8.95,8.95]$ and
$$
\inf\{|\mu_0-\mu_1(\theta_1)|;\theta_1\in\Theta_1\}=\inf\{|\mu_0-\mu_2(\theta_2)|;\theta_2\in\Theta_2\}=0\,.
$$
We can prove that $d_1<d_2$ and hence that the Bayes factor is consistent.
Moreover, when $\model_2$ is true with $g=1$ and $k=2$, then $\mu_0\approx [-4.90,12.99]$, and
$$
\inf\{|\mu_0-\mu_2(\theta_2)|;\theta_2\in\Theta_2\}=0
\quad\text{ but }\quad
\inf\{|\mu_0-\mu_2(\theta_1)|;\theta_1\in\Theta_1\}>0\,.
$$ 
We can prove that $\min(\alpha_1,\tau_1)>d_2$ and therefore that the Bayes factor is again consistent.

Finally, when the summary statistics $\feta^n$ is the larger vector made of the empirical quantiles of order
$10\%$, $40\%$, $60\%$ and $90\%$, the Bayes factor is obviously consistent.  However, the results obtained in
Fig.~\ref{fig:quantiles} are very similar to the ones obtained with only two empirical quantiles.
 
\begin{figure}[ht]
\centerline{\includegraphics[height=9truecm]{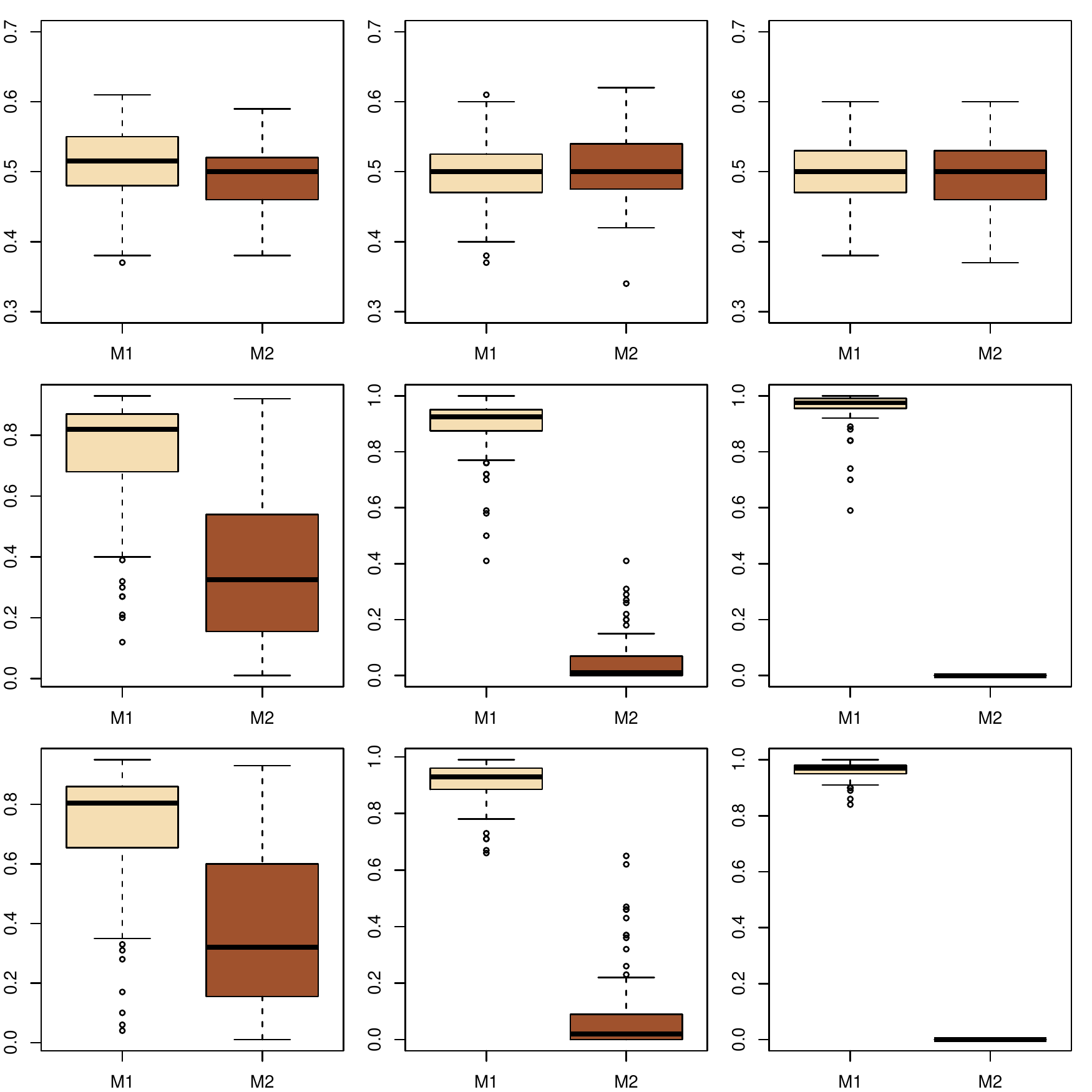}}
\caption{Comparison of the distributions of the posterior probabilities that
the quantile distribution data is from model $\model_1$ when the data is made
of 100 observations {\em (left column)}, 1000 observations {\em (central
column)} and 10,000 observations {\em (right column)} either from $\model_1$
{\em (M1)} or $\model_2$ {\em (M2)} when the summary statistics in the ABC
algorithm are made of the empirical quantile at level $10\%$ {\em (first row)},
the empirical quantiles at levels $10\%$ and $90\%$ {\em (second row)}, and the
empirical quantiles at levels $10\%$, $40\%$, $60\%$ and $90\%$ {\em (third
row)}, respectively. The boxplots rely on $100$ replicas and the ABC algorithms
are based on $10^4$ proposals ($5,000$ for each model) from the prior, with the
tolerance being chosen as the $1\%$ quantile on the distances.
}
\label{fig:quantiles}
\end{figure}

\subsection{Population genetics experiment}\label{sec:popgenX}

We now examine a Monte Carlo experiment that more directly relates to the genesis of ABC, namely population
genetics.  As in \cite{robert:cornuet:marin:pillai:2011}, we consider two populations (1 and 2) having diverged
at a fixed time $t'$ in the past and a third population (3) having diverged from one of those two populations
(models 1 and 2, respectively). Times are set to $t'=60$ generations for the first divergence and $t=30$
generations for the second one. The effective population size is assumed to be identical for all three
populations and equal to $N_e=60$. Recall that the effective size of a population is defined as the
size of an ideal (Wright-Fisher) population that would show the same behaviour as the population of interest,
in terms of loss of genetic variation due to random drift. \\
We assume we observed 50 diploid individuals per population genotyped at 5,
50 or 100 independent microsatellite loci, this number acting as a proxy to the sample size. These loci are
assumed to evolve according to the stepwise mutation model: when a mutation occurs, the number of repetitions
of the mutated gene increases or decreases by one unit with equal probability. For each configuration (defined
in terms of loci numbers), we generate 100 observations for which the mutation rate $\theta$ is common to all loci
and set to $0.005$.  In these experiments, both scenarios have a single parameter, the mutation rate $\theta$.  We
chose a uniform prior distribution $\mathcal{U}[0.0001,0.01]$ on this parameter $\theta$.  

For the ABC analysis, we use three
summary statistics associated to the $(\delta\mu)^2$ distances \citep{Goldstein:etal:1995,Cooper:etal:1999}.
Let $x_{l,i,j}$ be the repeated number of allele in locus $l=1,\ldots,n$ ($n=5,50,100$)
for individual $i=1,\dots,100$ (corresponding to $50$ diploid individuals) within population $j=1,2,3$.
The $(\delta\mu)^2$ distance between population $j_1$ and $j_2$, denoted by $(\delta\mu)^2_{j_1,j_2}$, is:
$$
(\delta\mu)^2_{j_1,j_2}=\frac{1}{n}\sum_{l=1}^n \left(\frac{1}{100}\sum_{i_1=1}^{100}x_{l,i_1,j_1}-
\frac{1}{100}\sum_{i_2=1}^{100}x_{l,i_2,j_2}\right)^2\,.
$$

Let us consider two copies of the locus $l$ with allele sizes $x_{l,i,j_1}$ and $x_{l,i',j_2}$, and assume that
the most recent time in the past for which they have a common ancestor, defined as the coalescence time
$\tau_{j_1,j_2}$, is known.  The two copies are then separated by a branch of gene genealogy of total length
$2\tau_{j_1,j_2}$.  As explained in \cite{Slatkin:1995}, according to the coalescent process, during that time
the number of mutations is a random variable distributed from a Poisson distribution with parameter
$2\mu\tau_{j_1,j_2}$.  Therefore, if the stepwise mutation model is adopted, we get (under models 1 and 2)
$$
\E^i\left\{\left(x_{l,i,j_1}-x_{l,i',j_2}\right)^2|\tau_{j_1,j_2}\right\}=2\theta\tau_{j_1,j_2}\,.
$$
In addition, if $j_1=1$ and $j_2=2$, we have (under models 1 and 2)
$$
\E\left\{\tau_{1,2}\right\}=\left(2N_e+t'\right)\,,
$$
and
$$
\E\left\{\left(x_{l,i,1}-x_{l,i',2}\right)^2\right\}=4N_e\theta+2\theta t'\,.
$$
Moreover,
\begin{align*}
\left(x_{l,i,1}-x_{l,i',2}\right)^2&=
\left(x_{l,i,1}-\frac{1}{100}\sum_{i_1=1}^{100}x_{l,i_1,1}+\frac{1}{100}\sum_{i_1=1}^{100}x_{l,i_1,1}\right.\\
&\left.+\frac{1}{100}\sum_{i_2=1}^{100}x_{l,i_2,2}-\frac{1}{100}\sum_{i_2=1}^{100}x_{l,i_2,2}-x_{l,i',2}\right)^2
\end{align*}
And
\begin{align*}
&\E\left\{\left(x_{l,i,1}-\frac{1}{100}\sum_{i_1=1}^{100}x_{l,i_1,1}\right)^2\right\}+
\E\left\{\left(x_{l,i',2}-\frac{1}{100}\sum_{i_2=1}^{100}x_{l,i_2,2}\right)^2\right\}+\\
&\quad
\E\left\{\left(\frac{1}{100}\sum_{i_1=1}^{100}x_{l,i_1,1}-\frac{1}{100}\sum_{i_2=1}^{100}x_{l,i_2,2}\right)^2\right\}=
4N_e\theta+2\theta t'\,.
\end{align*}
The coalescent process associated to the stepwise mutation model gives
$$
\E\left\{\left(x_{l,i,1}-\frac{1}{100}\sum_{i_1=1}^{100}x_{l,i_1,1}\right)^2\right\}=
\E\left\{\left(x_{l,i',2}-\frac{1}{100}\sum_{i_2=1}^{100}x_{l,i_2,2}\right)^2\right\}=2N_e\theta\,,
$$
and then
$$
\E\left\{(\delta\mu)^2_{1,2}\right\}=2\theta t'\,.
$$
We can apply the same type of reasoning to the other $(\delta\mu)^2$ distances, and
if $\mu_i$ denotes the mutation rate under model $i$, we get the results given
in Table \ref{tab:tab1}.

\begin{table}
\caption{\label{tab:deltamus} Theoretical expectations of
the $(\delta\mu)^2_{a,b}$ statistics under both models \label{tab:tab1}}
\begin{tabular}{l c c}
                                      & \mbox{Model 1} & \mbox{Model 2} \\
\hline $\E\left\{(\delta\mu)^2_{1,2}\right\}$ & $2\theta_1 t'$& $2\theta_2 t'$ \\
	$\E\left\{(\delta\mu)^2_{1,3}\right\}$ & $2\theta_1 t$ & $2\theta_2 t'$ \\
	$\E\left\{(\delta\mu)^2_{2,3}\right\}$ & $2\theta_1 t'$& $2\theta_2 t$ \\
\end{tabular}
\end{table}
 
Given the complexity of this genetic model, it provides a realistic example
of relevant statistics satisfying  the assumptions (\ref{A1})--(\ref{A5}). Let us consider
the associated statistics $(\delta\mu)^2_{1,2}$. This is an empirical mean of variables 
$$
Y_l=\left(\frac{1}{100} \sum_{i_1=1}^{100}x_{l,i_1,1} - 
\frac{1}{100}\sum_{i_2=1}^{100}x_{l,i_2,2}\right)^2\,,
$$
which are independent and identically distributed. Moreover, since for each couple $(i_1, i_2)$, 
$\left(x_{l,i_1,1}- x_{l,i_2,2}\right)$ is bounded by a Poisson random variable, then, under each model,
$Y_l$ has moments of all orders and $(100)^2Y_l \in \Z$.
Thus, using Theorem 2.2.1 of \cite{battacharya:rao}, we obtain that for all $s \geq 2$ and all $\theta_i$
\begin{equation}\label{Edg:Mi}
( 1 + |t|^{s}) \left| g_i(t|\theta_i)   - p_{n,s,\theta_i}( n^{1/2}( t - \mu_i(\theta_i))) \right| =  o(n^{(s-1)/2})
\end{equation}
where 
\begin{enumerate}
\item[(i)] the order
$o(n^{(s-1)/2})$ is uniform over $t\in \Z/(100^2n)$ and over compact subsets of $\Theta_i\subset \R^+$,
\item[(ii)] $p_{n,s,\theta_i}( n^{1/2}( t - \mu_i(\theta_i)))$ is the Edgeworth expansion of the 
density of $(\delta\mu)^2_{1,2}$:
$$
p_{n,s,\theta_i}( n^{1/2}\{ t - \mu_i(\theta_i) \} ) = n^{-1/2}\phi( n^{1/2}( t - \mu_i(\theta_i))/\sigma_i(\theta_i))  
+ o(n^{-1/2})
$$
for some $\sigma_i(\theta_i)>0$.

\end{enumerate}
 
If under the true distribution $Y_l$ also has at least $s\geq 2$ moments, then 
\begin{equation}\label{Edg:G0}
g_n(t) = n^{-1/2}\phi( n^{1/2}( t - \mu_0)/\sigma_0)  + o(n^{-1/2}), \quad \forall t\in \Z/(100^2n)
\end{equation}
for some $\mu_0\in \R$ and 
	$\sqrt{n} ( (\delta\mu)^2_{1,2} - \mu_0) \rightsquigarrow \mathcal N( 0, \sigma_0^2 ) $
thus (\ref{A1}) is satisfied with $v_n = \sqrt{n}$. Introducing 
	$E_n = \{ t \in \Z/( 100^2 n ) ; |t - \mu_0| \leq n^{-1/2} \gamma_\epsilon \}$, 
$G_n(E_n^c) < \epsilon$ from \eqref{Edg:G0} and \eqref{Edg:Mi} implies that, 
if $|\mu_i(\theta_i) - \mu_0| \leq \delta_0 n^{-1/2}$,
$g_i(t|\theta_i) \geq \delta g_n(t)$ for all $t \in E_n$ by choosing $\delta_0 $ 
small enough. Hence (\ref{A5}) is verified.  Moreover, if model $\model_i$ is compatible, 
	$\mu_1(\theta_1) =  2 \theta_1 t'$ and  $\mu_2(\theta_2) =  2 \theta_2 t'$, 
for $\mu_0 \in (0.0002t',0.02t')$, 
$$
\pi_i\left(S_{n,i}(u)\right) \sim u\,n^{-1/2}
$$
and (\ref{A4}) is satisfied with $d_i=1$. It is straightforward to verify (\ref{A3}).
Indeed, if model $\model_i$ is not compatible, choosing 
	$\epsilon_i > \inf\{ |\mu_0 - \mu_i(\theta_i)|; \theta_i \in ( 0.0001, 0.01) \}$,
for all $\theta_i \in ( 0.0001, 0.01)$, using \eqref{Edg:Mi} we get
$$
G_{i,n}[ |(\delta\mu)^2_{1,2} - \mu_i(\theta_i)  | >\epsilon_i |\theta_i] = o(n^{-(s-1)/2})
$$
uniformly in $\theta_i$ for all $s\geq 2$. If model $\model_i$ is compatible, using Markov inequality,
\begin{equation}
G_{i,n}[ |(\delta\mu)^2_{1,2} - \mu_i(\theta_i) | > x |\theta_i] \leq  
\frac{  \E_{\theta_i}\left[ |(\delta\mu)^2_{1,2} - \mu_i(\theta_i) |^4\right] }{  x^4 }
\end{equation}
and (\ref{A3}) is satisfied with $\alpha_i=4$. We can use the same arguments 
to show that assumptions (\ref{A1})--(\ref{A5}) holds also if
$\feta^n(\by)=((\delta\mu)^2_{1,3} , (\delta\mu)^2_{2,3})$. 

Table \ref{tab:deltamus} indicates that whatever model the data originates from (whether 
$\model_1$ or $\model_2$), the Bayes factor based only on the distance $(\delta\mu)^2_{1,2}$ as the
summary statistic does not converge.  Indeed, if $\theta_1=\theta_2$,
we get the same expectation on the first line of Table \ref{tab:deltamus}. The same occurs when only
$(\delta\mu)^2_{1,3}$ (resp. $(\delta\mu)^2_{2,3}$) is used. Indeed, in that case, if $\theta_1=2\theta_2$ (resp.
$2\theta_1=\theta_2$) we get the same expectation on the second (resp., the third) row of Table \ref{tab:deltamus}.
Now, if either two or three of the distances are used, the Bayes factors do converge. Indeed, in these settings,
no value of $\theta_1$ and $\theta_2$ can produce equal expectations.

Fig.~\ref{fig:popgen} shows how the empirical results confirm this theoretical analysis. Even the medium 
case of 50 loci indicates whether the use of the corresponding summary statistic(s) is valid or not.
Under both models, the ABC computations have been performed using the DIY-ABC software \citep{Cornuet:etal:2008}.

\begin{figure}[ht]
\centerline{\includegraphics[width=9truecm]{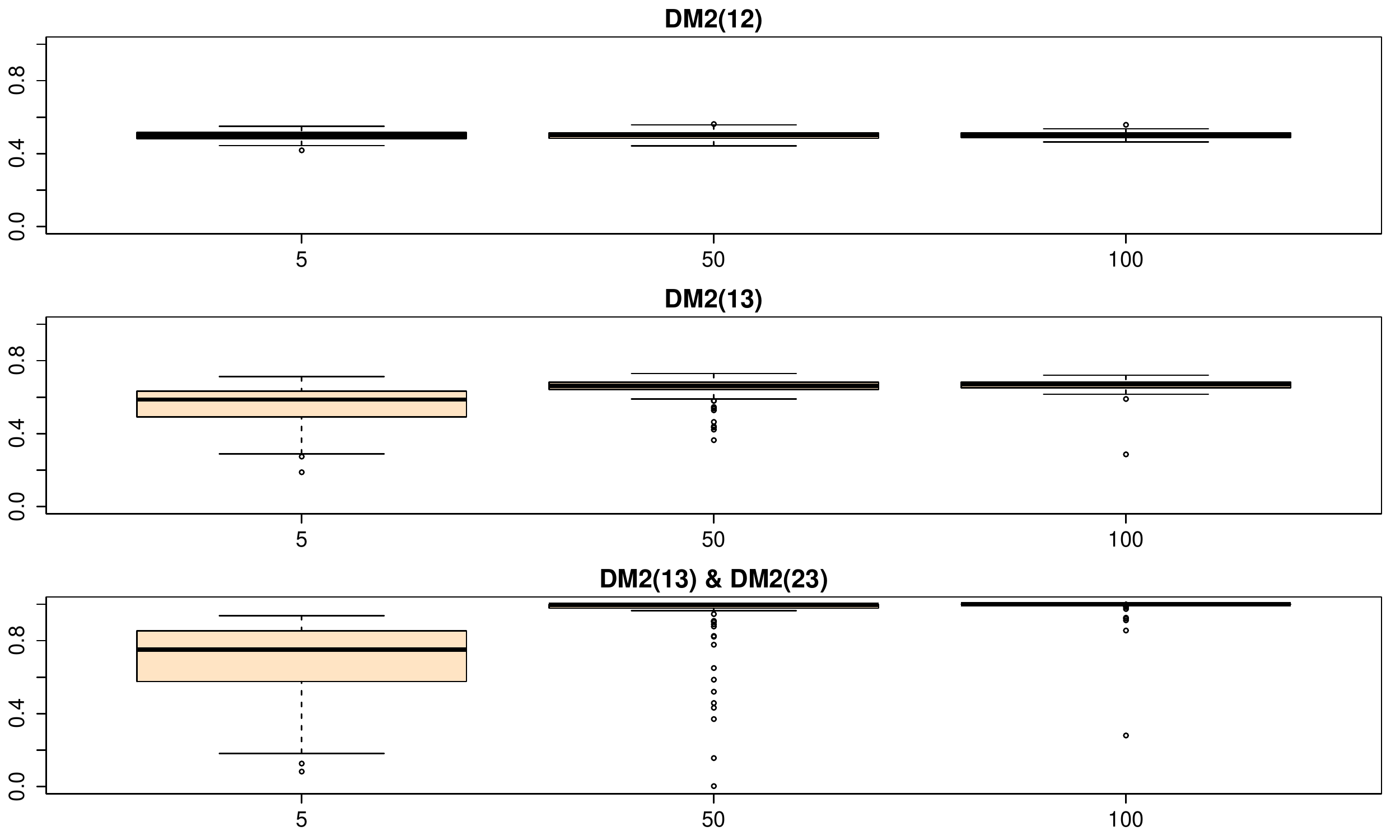}}
\caption{Comparison of the distributions of the posterior probabilities that the data is from model 1
for 5, 50 and 100 loci {\em (left, centre, right, resp.)} when the summary statistic in the ABC algorithm is 
made of different $(\delta\mu)^2$ distances. The ABC algorithm uses $2\times10^5$ proposals ($10^5$
for each model) from the prior and selects the tolerance $\epsilon$ as the $0.5\%$ distance quantile.
}
\label{fig:popgen}
\end{figure}

\section{Checking for relevant statistics}
\label{sec:checkstat}

\subsection{A practical procedure}

While Theorem 1 operates in an asymptotic and theoretical framework, it is
nonetheless possible to find a methodological consequence from this
characterisation of consistent summary statistics for testing. This result
states that the summary statistic $\feta^n$ is not consistent (and thus
unacceptable) for testing between models when both models are compatible with
$\feta^n$, in other words when
$$
\inf\{ |\mu_1(\theta_1)- \mu_0| ; \theta_1 \in \Theta_1 \} = \inf\{ |\mu_2(\theta_2)- \mu_0| ; 
\theta_2 \in \Theta_2 \}=0\,.
$$
Based on this our asymptotic result, we propose to run a practical check of the
relevance (or non-relevance) of $\feta^n$.  The null hypothesis
of this test is expressed as both models are
compatible with the statistic $\feta^n$. The testing procedure then provides estimates
of the mean of $\feta^n$ under each model and checks whether or not those means are equal. For the
sake of clarity, we assume without loss of generality that $\model_1$ is the
true model (recall that it is enough to have this model compatible with the statistic $\feta^n$), so
checking the relevance of $\feta^n$ means testing for
 $$H_0 : \inf\{ |\mu_2(\theta_2)- \mu_0| ; \theta_2 \in \Theta_2 \}=0$$
against 
 $$ H_1 :\inf\{ |\mu_2(\theta_2)- \mu_0| ; \theta_2 \in \Theta_2 \}>0 .$$

Corollary \ref{corollary:1} implies that, when model $\model_i$ is compatible
with $\feta^n$, the predictive value of the summary statistic,
$\E^\pi\left[\feta^n(\by^{\rm{new}})|\feta^n(\by) \right]$, is approximately
equal to $\mu_0$ ($\by$ denotes the observed summary statistic):
\begin{equation*}
\begin{split}
  \E^\pi\left[\feta^n(\by^{\rm{new}})|\feta^n, \model_i \right] &=\int  
	t g_{i,n}(t | \theta_i) dt d\pi_i( \theta_i| \feta^n(\by) )  \\
  &= \int_{\Theta_i} \mu_i(\theta_i) d\pi_i( \theta_i| \feta^n(\by) ) \\
  &= \mu_0 + \int_{\Theta_i} ( \mu_i(\theta_i) - \mu_0) d\pi_i( \theta_i| \feta^n(\by))\,.
\end{split}
\end{equation*}
When $|\mu_i(\theta_i)|$  is bounded on $\Theta_i$ (for instance when $\Theta_i$ is compact 
and $\mu_i(\cdot)$ is continuous)
  $$\int_{\Theta_i} ( \mu_i(\theta_i) - \mu_0) d\pi_i( \theta_i| \feta^n(\by) ) = o_p(1). $$
Thus, under the null (non-relevance of $\feta^n$), we have 
$$  
\E^\pi\left[\feta^n(\by^{\rm{new}})|\feta^n(\by),  \model_1 \right] =  
\E^\pi\left[\feta^n(\by^{\rm{new}})|\feta^n(\by),  \model_2 \right] + o_p(1) = \mu_0 + o_p(1)
$$
and the proximity of both predictive values indicates that the statistic $\feta^n$ is not discriminant. 

To quantify what this notion of proximity means we advocate using the following 
practical procedure. Under each model 
$\model_i$, $i=1,2$, run an ABC sample producing a sample $\theta_{i,l}, l=1, \cdots , L$ from 
the approximate posterior distribution of $\theta_i$ given $\feta^n(\by)$. Note that $L$ can be 
chosen to be arbitrarily large. For each value $\theta_{i,l}$, generate $\by_{i,l}\sim F_{i,n}(\cdot|\psi_{i,l})$,
derive $\feta^n(\by_{i,l})$ and compute 
$$
\hat{\mu_i} = \frac{1}{ L } \sum_{l=1}^L\feta^n(\by_{i,l}) , \quad i = 1,2\,.
$$
Conditionally on $\feta^n(\by)$, we have 
$$
\sqrt{L}( \hat{\mu_i} - \E^\pi\left[ \mu_i(\theta_i)| \feta^n(\by) \right]) \rightsquigarrow \mathcal N( 0, V_i),
$$
for some $V_i$, as $L$ goes to infinity. Therefore, we propose to test for a common mean
$$
H_0 : \hat{\mu_1} \sim \mathcal N(\mu_0, V_1)\,, \hat{\mu_2} \sim \mathcal N(\mu_0, V_2)
$$
against the alternative of different means
$$
H_1 : \hat{\mu_i} \sim \mathcal N(\mu_i, V_i), \quad \mbox{ with } \mu_1 \neq \mu_2\,.
$$
This test is implemented using the fact that asymptotically the decision statistic
$$
(\hat\mu_1-\hat\mu_2)^\text{T}(V_1+V_2)^{-1}(\hat\mu_1-\hat\mu_2)
$$
converges to a chi-squared distribution even in the case the covariance matrices $V_1$ and
$V2$ are estimated by convergent estimators, e.g.~empirical covariances.
If the null hypothesis $H_0$ cannot be rejected,
we conclude that the statistic $\feta^n(\by)$ is not adequate for model choice.

\subsection{Gaussian versus Laplace distributions}

In the case of the normal versus Laplace toy problem, we ran a hundred evaluations based on three and four
statistics, i.e.~the empirical mean, median and variance, without and with the empirical mad. 
The two choices of the summary statistic vector led to two different ABC approximations of the posterior
distribution. Under each model, the ABC procedure is based on a fixed reference table
of $5\times 10^4$ proposals from the prior $\theta\sim \mathcal{N}(0,4)$
and the respective model, and it selects the tolerance $\epsilon$ as the $1\%$ quantile of
the deduced simulation distances. Then, under each model, we get a $500$-sample as an approximation
of the posterior distributions. For each of those values, we simulated samples from the models,
with the same size as the original sample, producing samples $\feta^n(\by_{i,l})$, and we
derived from those samples $\chi^2$ tests about the equality of the means. Fig.~\ref{fig:verif-gl} 
evaluates the impact of including the empirical mad within those summary statistics on the result of the $\chi^2$
test. The result of this simulation experiment (based on 100 replications) is quite satisfactory in that
in approximately $95\%$ of cases the difference between the two empirical means falls within the null hypothesis
acceptance interval associated with a $5\%$ error when the empirical mad is not included, therefore concluding on the
inappropriateness of the summary statistics to conduct the ABC model comparison.
On the opposite, the difference always is outside the null hypothesis acceptance interval when the empirical mad is included.

\begin{figure}[ht!]
\centerline{
\includegraphics[width=6truecm]{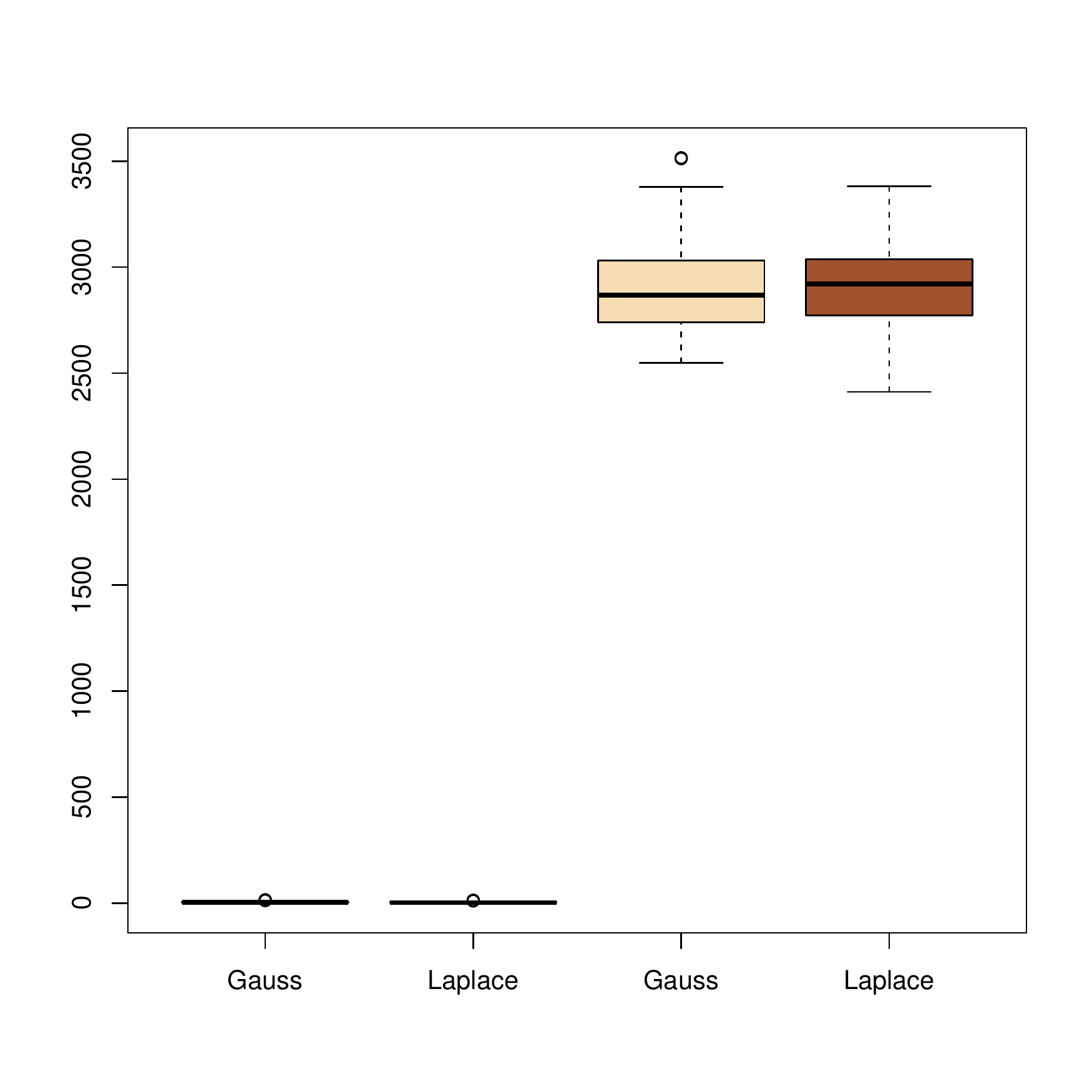}
\includegraphics[width=6truecm]{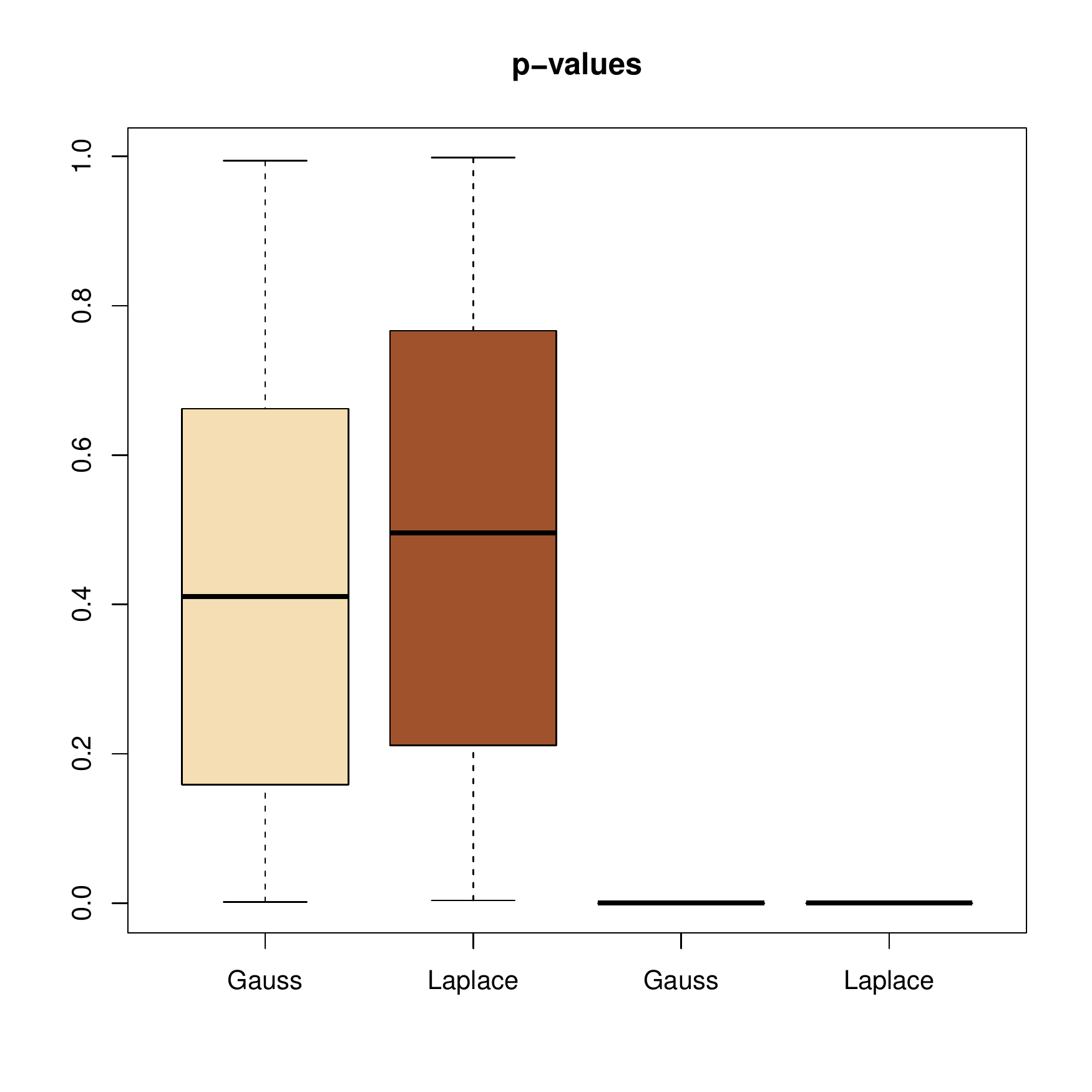}
}
\caption{Boxplot representation of the variability of the test statistics ({\it left}) and 
corresponding p-values ({\it right}) comparing the expectation of two vectors of summary statistics: 
the empirical mean, median and variance (left block within each graph) and the previous three ones 
plus the empirical mad (right block within each graph). 
} 
\label{fig:verif-gl}
\end{figure}

\subsection{Population genetics experiment}

For the population genetics example, we ran a comparison experiment between the
case when we use $(\delta\mu)^2_{1,2}$ as sole summary statistic and when we
use instead the vector $\left((\delta\mu)^2_{1,3},(\delta\mu)^2_{2,3}\right)$.
The results are presented in Fig. \ref{fig:verif-popgen}. For each of the $100$
points used for the boxplots, the data is made of $n=100$ loci and the
$t$-statistics are based on $500$ samples under each model.  In fact, the ABC
algorithm relies on a fixed reference table of $10^5$ proposals from the prior
and the respective model, and it selects the tolerance $\epsilon$ as the
$0.5\%$ quantile of the deduced simulation distances.  As for the previous
example, the result of this simulation experiment is quite satisfactory.  It
highlights the ability of our empirical procedure to detect inappropriate
summary statistics.  We can truly compare the expectation of a summary
statistics under both models using parameter values drawn from the ABC
posterior approximation.

\begin{figure}[ht!]
\centerline{
\includegraphics[width=6truecm]{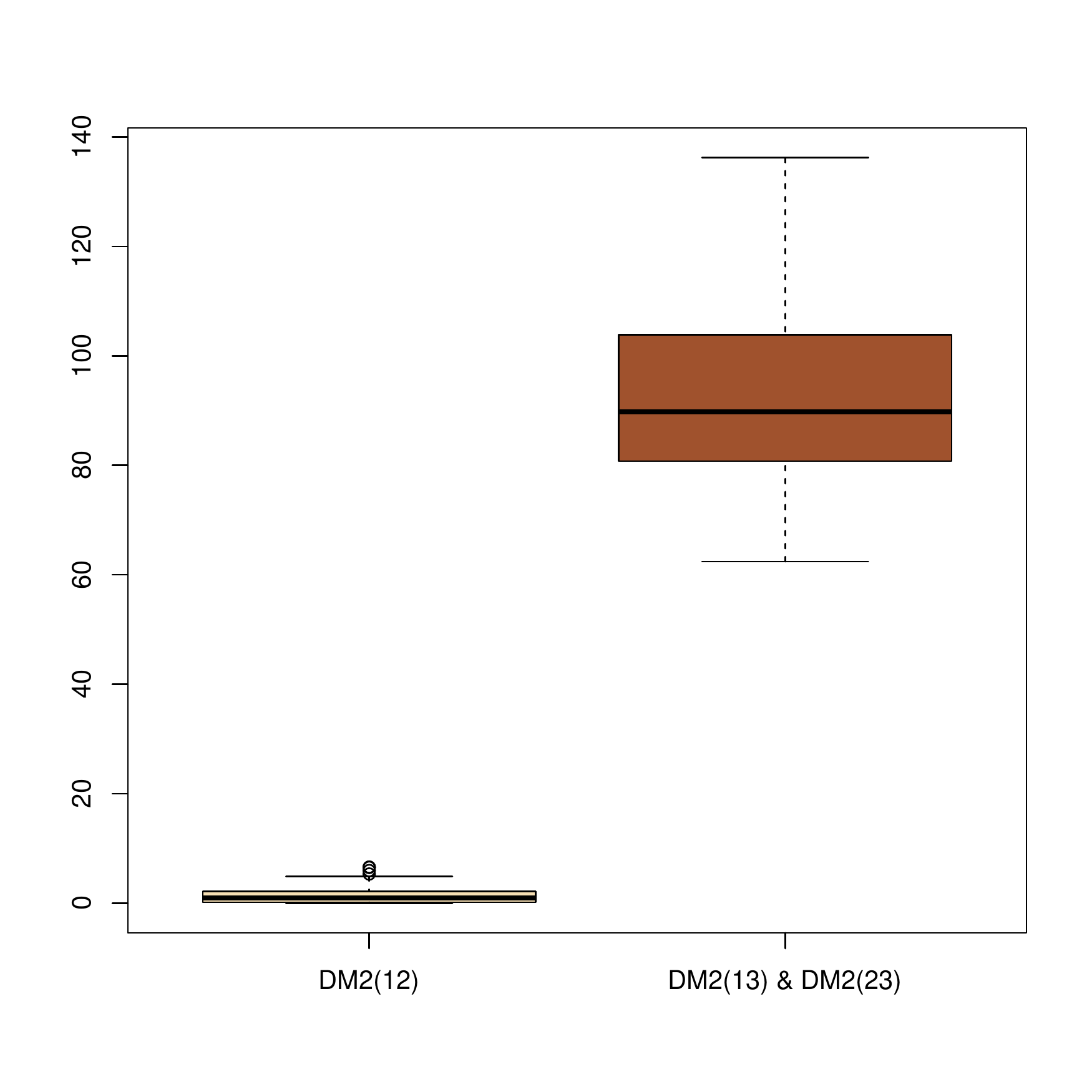}
\includegraphics[width=6truecm]{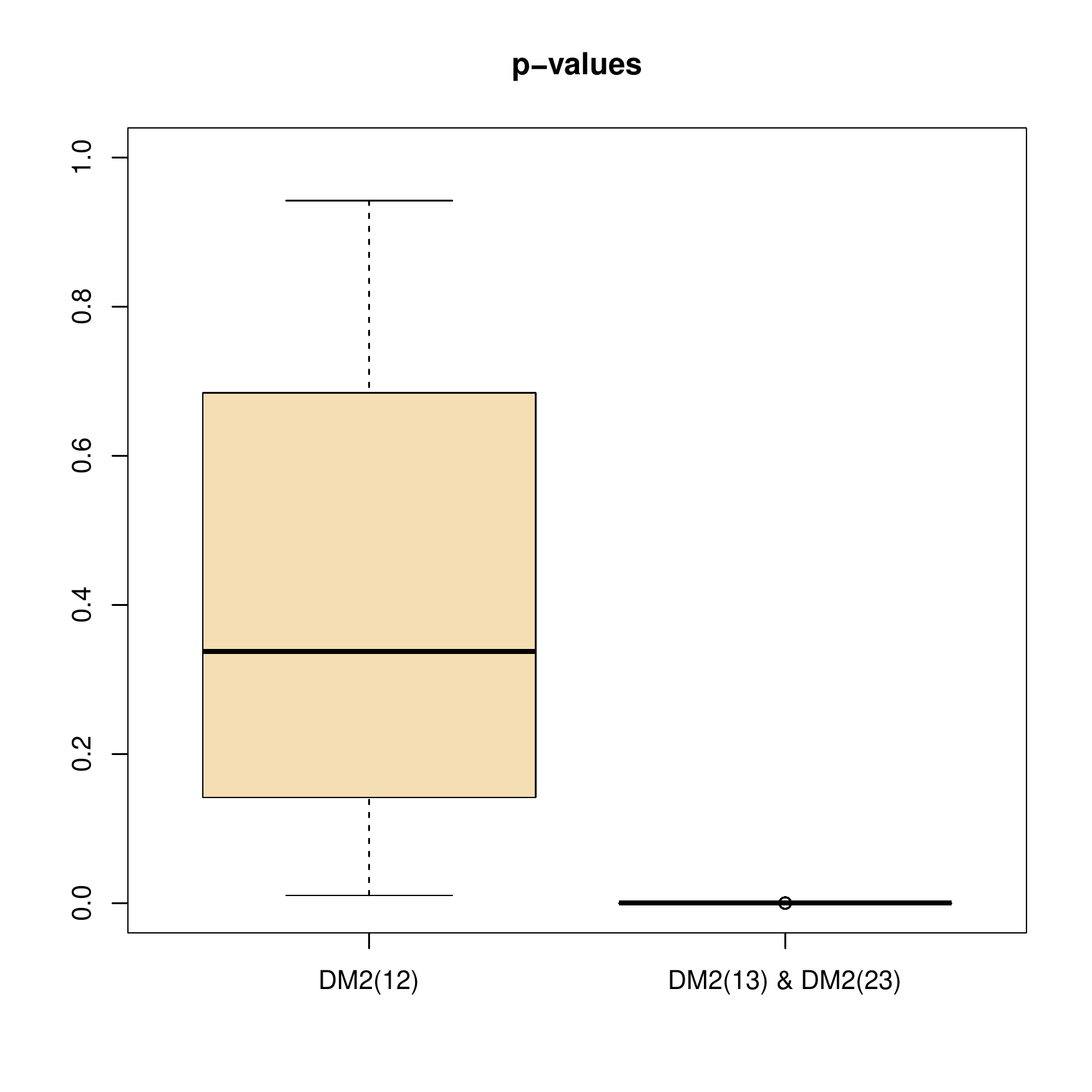}
}
\caption{Boxplot representation of the variability of the test statistics ({\it left}) and 
corresponding p-values ({\it right}) comparing the mean of the summary statistics $(\delta\mu)^2_{1,2}$ 
and $\left((\delta\mu)^2_{1,3},(\delta\mu)^2_{2,3}\right)$.} 
\label{fig:verif-popgen}
\end{figure}

\section{Discussion}\label{seccon}

This paper has produced sufficient conditions for a summary statistics to produce a consistent or an
inconsistent Bayesian model choice. It thus brings an answer to the question raised in
\cite{robert:cornuet:marin:pillai:2011}, which was warning the ABC community about the potential pitfalls of an
uncontrolled use of ABC approximations to Bayes factors.  The central condition that the true asymptotic mean
of the summary statistic should not be recovered under the wrong model if model choice is to take place (in a
convergent manner) is both natural, in that the asymptotic normality implies that only first moments matter,
and fundamental, in that it drives the choice of summary statistics in practical settings, first and foremost
for the ABC algorithm. Indeed, Theorem \ref{thm:BFcon} implies that estimation statistics should not be used in
ABC algorithms aiming at model comparison, unless their expectation can be shown to differ under both models.
This means that (a) different statistics should be used for estimation and for testing and (b) that they should
not be mixed in a single summary statistic. Note that the distinction differs from the sufficient versus
ancillary opposition found in classical statistics \citep{cox:hinkley:1974} in that it is enough that the
summary statistic $\feta^n$ has a different asymptotic mean under both models. In addition, and as shown in the
normal-Laplace example in Section \ref{subsec:GaLap}, some ancillary statistics may not be appropriate for
testing.

At a methodological level, the classification of summary statistics resulting from the present study is
paramount: when comparing models with a given range of potential summary statistics, the expectations of the
various summary statistics can be evaluated by simulation under all models. For instance, in ABC settings, the
production of pseudo-data is a requirement for the implementation of the method; it is therefore quite
straightforward to test via a preliminary experiment whether the condition of Theorem \ref{thm:BFcon} holds.

Neither the final choice of summary statistics as in
\cite{fearnhead:prangle:2012}, nor the comparison with alternative model
comparisons techniques such as $\text{ABC}_\mu$
\citep{ratmann:andrieu:wiujf:richardson:2009} are covered in the current paper.
These obviously are issues worth investigating and they constitute seeds for
future development in the area.

\section*{Acknowledgements} 
Part of this work was done when the second author (NSP) was visiting Dauphine and CREST and he thanks 
both institutions for their warm hospitality. All authors but the second author are partially supported 
by the Agence Nationale de la Recherche (ANR, 212, rue de Bercy 75012 Paris) through the 2009--2013 
projects {\sf Bandhits} and {\sf Emile}.  The second author is partially supported by the NSF grant 
1107070. The authors are grateful to the whole editorial panel for their supportive and constructive comments
throughout the editorial process. Discussions with Dennis Prangle at various stages of the paper were quite
helpful.

\appendix
\section*{Appendix 1}
\subsection*{Proof of Lemma \ref{thm:normasympt}}

\small
Recall that $G_n$ is the true distribution of $\feta^n$.
Let us first assume that $\inf \{ |\mu_0 -\mu_i(\theta_i)|; \theta_i \in \Theta_i \} = 0$ and let  $S_{n,i}$ be 
defined as in (\ref{A4}). Note that from (\ref{A1}), for all $\delta >0$, there exists $M_\delta>0$ such that 
for $n$ large enough
\begin{equation}\label{devGN}
G_n(|\feta^n-\mu_0|>M_\delta /v_n ) <\delta .
\end{equation}
and note that $M_\delta $ goes to infinity as $\delta$ goes to 0. Let $\epsilon>0$ and consider the set $E_n$, and the positive constants $U_\epsilon$ and $\delta_\epsilon$ defined in (\ref{A5}).
From \eqref{eqn:ML} we have  for all $\feta^n \in E_n$,
\begin{equation*}\label{eqn:misecond}
\begin{split}
m_i(\feta^n)& \geq \int_{S_{n,i}(U_\epsilon)} g_i(\feta^n |
\theta_i)\,\pi_i(\theta_i)\,\text{d}\theta_i\, \\
&\geq \delta_\epsilon g_n(\feta^n)\pi_i( S_{n,i}(U_\epsilon) ) 
 \geq c_\epsilon g_n(\feta^n)v_n^{-d_i}
\end{split}
\end{equation*}
for some positive constant $c_\epsilon$, where the second inequality follows from the definition of $E_n$ 
and the last from (\ref{A4}). Formally, since $G_n( E_n) \geq 1 - \epsilon $, there exists 
$c_\epsilon>0$ such that for $n$ large enough
\begin{eqnarray} \label{lb}
G_n\left[ m_i(\feta^n)  \geq 
c_\epsilon\, g_n(\feta^n) v_n^{-d_i} \right]\geq  1 - \epsilon.
\end{eqnarray}

We now obtain an upper bound for $m_i(\feta^n)$. Using \eqref{eqn:ML} we write,
\begin{eqnarray*}
m_i(\feta^n) &=& \int_{\mathcal F_{n,i}} g_i(\feta^n |\theta_i)\,\pi_i(\theta_i)\,d\theta_i + \int_{\mathcal F_{n,i}^c} g_i(\feta^n |\theta_i)\,\pi_i(\theta_i)\,d\theta_i \;.
\end{eqnarray*}
As before fix $\delta>0$ and let $M_\delta$ be defined as in \eqref{devGN}.
Applying Markov inequality and Fubini's theorem,  we obtain that, for all $\epsilon>0$, 
\begin{equation} \label{Fnc}
\begin{split}
& G_n\Big(\int_{\mathcal F_{n,i}^c} g_i(\feta^n |\theta_i)\,\pi_i(\theta_i)\,d\theta_i >\epsilon\, g_n(\feta^n)\,\pi_i(S_{n,i}) \Big) \\
&\leq G_n[ |\feta^n -\mu_0| > M_\delta v_n^{-1} ] \\
&  +  \int_{\mathcal F_{n,i}^c }\int_{v_n|t - \mu_0|\leq M_\delta} \frac{1}{\epsilon\,g_n(t)\,\pi_i(S_{n,i})} 
g_n(t) g_i(t|\theta_i) dt  \,\pi_i(\theta_i)\,d\theta_i   \\
&\leq G_n[ |\feta^n -\mu_0| > M_\delta v_n^{-1} ] \\
&  +  \int_{\mathcal F_{n,i}^c }\frac{1}{\epsilon\,\pi_i(S_{n,i})} \int_{\R^d} g_i(t|\theta_i) dt  \,\pi_i(\theta_i)\,d\theta_i   \\
& \leq \delta + \frac{\pi(\mathcal F_{n,i}^c )}{\epsilon \pi_i(S_{n,i})} \leq 2 \delta,
\end{split}
\end{equation}
given the conditions imposed by (\ref{A3}) and (\ref{A4}), when $n$ is large enough.

We can represent $\mathcal F_{n,i}$ as a finite disjoint union of the following sets:
\begin{eqnarray*}
\mathcal{F}_{n,i} &=& \bigcup_{j=0}^{J_n + 1} \mathcal{H}_j, \quad J_n = J_0 v_n, \quad \mathrm{for} \,\, \mathrm{some} \,\, J_0 \in \mathbb{N} \,, \\
\mathcal{H}_j &=&S_{n,i}((j+1) M_\delta) \cap S_{n,i}( jM_\delta)^c,  \quad j \leq J_n \,, \\
\mathcal{H}_{J_n+1} &=&\mathcal{F}_{n,i} \cap S_{n,i}^c(M_\delta J_n) \;.
\end{eqnarray*}
Now we have
\begin{equation}\label{decomp:j}
 \int_{\mathcal F_{n,i}} g_i(\feta^n |\theta_i)\,\pi_i(\theta_i)\,d\theta_i =  \sum_{j=0}^{J_n+1} \int_{\mathcal{H}_j } g_i(\feta^n |\theta_i)\,\pi_i(\theta_i)\,d\theta_i \;.
 \end{equation} 
 Set $c_0 = \sum_{j=2}^\infty j^{-( \alpha_i - d_i)/2} < + \infty $ since $\alpha_i > 2+ d_i$. Then, 
 if $j =0$, $\mathcal{H}_0 =S_{n,i}( M_\delta)$ and if $K$ is a constant such that $K>d_i$ we obtain
\begin{eqnarray}
G_n\left[ \int_{S_{n,i}( M_\delta)} g_i(\feta^n |\theta_i)\,\pi_i(\theta_i)\,d\theta_i > M_\delta^K g_n(\feta^n)v_n^{-d_i }\right] &\leq &  \frac{1}{ M_\delta^K v_n^{-d_i} } \pi_i(S_{n,i}(M_\delta)) + \delta \nonumber \\
&=& O( M_\delta^{d_i-K}) + \delta, \label{eqn:Mdeltadi}
\end{eqnarray}
where the last inequality follows from \eqref{eqn:A5e2} in (\ref{A4}). Since $\limsup_{\delta \rightarrow 0}M_\delta = \infty$, the bound in \eqref{eqn:Mdeltadi} goes to $0$ as $\delta$ goes to zero. 
Using (\ref{A3})  and following exactly the same argumentation as for \eqref{Fnc},
i.e.~Markov inequality and Fubini's theorem, we obtain that for $0<j \leq J_n$,
\begin{equation}\label{Fncbis}
\begin{split}
&
G_n \Big( \Big\{\int_{\mathcal{H}_j} g_i(\feta^n |\theta_i)\,\pi_i(\theta_i)\,d\theta_i > M_\delta^K g_n(\feta^n)v_n^{-d_i }j^{(d_i - \alpha_i)/2}\Big\} \cap \Big\{ v_n|\feta^n -\mu_0 | \leq M_\delta/2 \Big\} \Big) \\
&\leq  \frac{ c_0v_n^{d_i} }{ M_\delta^Kj^{(d_i - \alpha_i)/2} } \int_{\mathcal{H}_j}G_{i,n}\left( |\feta^n -\mu( \theta_i) | > (j-1/2)M_\delta v_n^{-1} |\theta_i \right) \,\pi_i(\theta_i)\,d\theta_i \\
&\lesssim  M_\delta^{d_i-\alpha_i-K} j^{(d_i-\alpha_i)/2} 
\end{split}
\end{equation}
for $n$ large enough, and similarly 
\begin{equation} \label{bound:Jn}
\begin{split}
& G_n\Big(\int_{\mathcal{H}_{J_n+1} } g_i(\feta^n |\theta_i)\,\pi_i(\theta_i)\,d\theta_i> g_n(\feta^n) v_n^{-d_i } \Big)\\
&\leq   v_n^{d_i}  \int_{\mathcal{H}_{J_n+1} }G_{i,n}\left( |\feta^n -\mu( \theta_i) | > J_0/2 |\theta_i \right) \,\pi_i(\theta_i)\,d\theta_i \\
   &  \qquad+ 
G_{n}\left( v_n|\feta^n -\mu_0 | >M_\delta \right) \\
&\leq   \delta + O( v_n^{d_i -\alpha_i}) \leq 2 \delta,
\end{split}
\end{equation}
for $n$ large enough, under assumption (\ref{A4}).
Combining the above inequalities \eqref{Fnc}, \eqref{Fncbis}, and \eqref{bound:Jn}
with \eqref{decomp:j}, we obtain for $n$ large enough,
\begin{equation*}
\begin{split}
 & G_n\Big(\int_{\mathcal F_{n,i}} g_i(\feta^n |\theta_i)\,\pi_i(\theta_i)\,d\theta_i > (2M_\delta^K + 1) g_n(\feta^n) v_n^{-d_i}\Big) \\
 & \leq   G_{n}\left( v_n|\feta^n -\mu_0 | > \frac{1}{2}M_\delta  \right)  + O(M_\delta^{d_i-K} )
\end{split}
\end{equation*}
which can be made arbitrarily small by choosing $\delta $ small enough. Combining the above with  \eqref{Fnc} implies that
$$
\int_{\Theta_i}\frac{ g_i(\feta^n |\theta_i)}{ g_n(\feta^n) }\pi_i(\theta_i)\,d\theta_i = O_{\P^n}( v_n^{-d_i}  ) \;.
$$
The above estimate together with the lower bound obtained in \eqref{lb}
proves the first claim (Equation \eqref{th1:1}) of Lemma \ref{thm:normasympt}.

Now suppose $\inf\{ |\mu_i(\theta_i) - \mu_0| ; \theta_i \in \Theta_i \} >  0$. 
Then  there exists $j_0 >0$ such that $S_{n,i}( j\,v_n) = \emptyset$ for all $j \leq j_0$.
A computation identical to \eqref{bound:Jn}, together with \eqref{Fnc}, yields that for 
all sequences $w_n$ going to infinity and all $\epsilon>0$,
{\small \begin{eqnarray*}
&          & G_n\Big(\int_{\mathcal F_{n,i}} g_i(\feta^n |\theta_i)\,\pi_i(\theta_i)\,d\theta_i >  g_n(\feta^n)(\epsilon v_n^{-\tau_i }+ w_nv_n^{-\alpha_i})\Big) \\
& \leq & G_{n}\left( v_n|\feta^n -\mu_0 | >M_\delta  \right)  + \frac{ \pi(\mathcal F_{n,i}^c )}{ \epsilon v_n^{-\tau_i }} \\ 
&          & + \frac{ v_n^{\alpha_i} }{w_n \epsilon } \int_{\mathcal F_{n,i}} G_{i,n} \left( |\feta^n - \mu_i(\theta_i)| > j_0 v_n/2  \right)  \,\pi_i ( \theta_i) \,d\theta_i\\
& \leq     & 2  \delta,
\end{eqnarray*}}
for $n$ large enough. This proves the second claim (Equation \eqref{th1:2}) of Lemma
\ref{thm:normasympt}.\hspace{1 cm} $\Box$


\normalsize

\begin{thebibliography}{}

\bibitem[\protect\citeauthoryear{Allingham, King and Mengersen}{
  Allingham et~al.}{2009}]{allingham:king:mengersen:2009}
Allingham, D., R.A.R.~King, K.L.~Mengersen (2009).
\newblock Bayesian estimation of quantile distributions.
\newblock {\em Statistics and Computing\/}~{\em 19\/}(2), 189--201.

\bibitem[\protect\citeauthoryear{Bhattacharya and Rao}{Bhattacharya and
  Rao}{1986}]{battacharya:rao}
Bhattacharya, R.~N. and R.~R. Rao (1986).
\newblock {\em Normal Approximation and Asymptotic Expansions}.
\newblock New-York: Wiley Series in Probability and Mathematical Statistics.

\bibitem[\protect\citeauthoryear{Cooper, Amos, Bellamy, Siddiqui, Frodsham,
  Hill, and Rubinsztein}{Cooper et~al.}{1999}]{Cooper:etal:1999}
Cooper, G., W.~Amos, R.~Bellamy, M.~Siddiqui, A.~Frodsham, A.~Hill, and
  D.~Rubinsztein (1999).
\newblock An empirical exploration of the $(\delta\mu)^2$ genetic distance for
  213 human microsatellite markers.
\newblock {\em American Journal of Human Genetics\/}~{\em 65\/}(4), 1125--1133.

\bibitem[\protect\citeauthoryear{Cornuet, Santos, Beaumont, Robert, Marin,
  Balding, Guillemaud, and Estoup}{Cornuet et~al.}{2008}]{Cornuet:etal:2008}
Cornuet, J.-M., F.~Santos, M.~Beaumont, C.~Robert, J.-M. Marin, D.~Balding,
  T.~Guillemaud, and A.~Estoup (2008).
\newblock Inferring population history with diyabc: a user-friendly approach to
  approximate {B}ayesian computation.
\newblock {\em Bioinformatics\/}~{\em 24\/}(23), 2713--2719.

\bibitem[\protect\citeauthoryear{Cox and Hinkley}{Cox and
  Hinkley}{1994}]{cox:hinkley:1974}
Cox, D. and D.~Hinkley (1994).
\newblock {\em Theoretical Statistics}.
\newblock Chapman \& Hall.

\bibitem[\protect\citeauthoryear{Dickey and Gunel}{Dickey and
  Gunel}{1978}]{dickey:gunel:1978}
Dickey, J. and E.~Gunel (1978).
\newblock Bayes factors from mixed probabilities.
\newblock {\em Journal of the Royal Statistical Society, Series B\/}~{\em 40\/}(1), 43--46.

\bibitem[\protect\citeauthoryear{Didelot, Everitt, Johansen, and
  Lawson}{Didelot et~al.}{2011}]{didelot:everitt:johansen:lawson:2011}
Didelot, X., R.~Everitt, A.~Johansen, and D.~Lawson (2011).
\newblock Likelihood-free estimation of model evidence.
\newblock {\em Bayesian Analysis\/}~{\em 6\/}(1), 1--28.

\bibitem[\protect\citeauthoryear{Diggle and Gratton}{Diggle and
  Gratton}{1984}]{diggle:graton:1984}
Diggle, P. and R.~Gratton (1984).
\newblock Monte Carlo methods of inference for implicit statistical models.
\newblock {\em Journal of the Royal Statistical Society, Series B\/}~{\em 46\/}(2), 193--227.

\bibitem[\protect\citeauthoryear{Fearnhead and Prangle}{Fearnhead and
  Prangle}{2012}]{fearnhead:prangle:2012}
Fearnhead, P. and D.~Prangle (2012).
\newblock {Semi-automatic approximate {B}ayesian computation}.
\newblock {\em Journal of the Royal Statistical Society, Series B\/}~(with discussion)~{\em 74\/}(3), 
419-–474.

\bibitem[\protect\citeauthoryear{Ghosal and van~der Vaart}{Ghosal and van~der
  Vaart}{2007}]{gvdv:06}
Ghosal, S. and A.~van~der Vaart (2007).
\newblock {Convergence rates of posterior distributions for non iid
  observations}.
\newblock {\em Annals of Statistics\/}~{\em 35\/}(1), 192--225.

\bibitem[\protect\citeauthoryear{Goldstein, Linares, Cavalli-Sforza, and
  Feldman}{Goldstein et~al.}{1995}]{Goldstein:etal:1995}
Goldstein, D., A.~Linares, L.~Cavalli-Sforza, and M.~Feldman (1995).
\newblock An evaluation of genetic distances for use with microsatellite loci.
\newblock {\em Genetics\/}~{\em 139\/}(1), 463--471.

\bibitem[\protect\citeauthoryear{Grelaud, Marin, Robert, Rodolphe, and
  Tally}{Grelaud et~al.}{2009}]{grelaud:marin:robert:rodolphe:tally:2009}
Grelaud, A., J.-M. Marin, C.~Robert, F.~Rodolphe, and F.~Tally (2009).
\newblock Likelihood-free methods for model choice in {G}ibbs random fields.
\newblock {\em Bayesian Analysis\/}~{\em 3(2)}, 427--442.

\bibitem[\protect\citeauthoryear{Haynes, MacGillivray, and Mengersen}{Haynes
  et~al.}{1997}]{haynes:macgillivray:mengersen:1997}
Haynes, M.~A., H.~L. MacGillivray, and K.~L. Mengersen (1997).
\newblock Robustness of ranking and selection rules using generalised
  {$g$-and-$k$} distributions.
\newblock {\em Journal of Statistical Planning and Inference\/}~{\em 65\/}(1), 45--66.

\bibitem[\protect\citeauthoryear{Marin, Pudlo, Robert, and Ryder}{Marin
  et~al.}{2011}]{marin:pudlo:robert:ryder:2011}
Marin, J., P.~Pudlo, C.~Robert, and R.~Ryder (2012).
\newblock Approximate {B}ayesian computational methods.
\newblock {\em Statistics and Computing\/}~{\em 21\/}(2), 289--291.

\bibitem[\protect\citeauthoryear{Pritchard, Seielstad, Perez-Lezaun, and
  Feldman}{Pritchard et~al.}{1999}]{pritchard:seielstad:perez:feldman:1999}
Pritchard, J., M.~Seielstad, A.~Perez-Lezaun, and M.~Feldman (1999).
\newblock Population growth of human {Y} chromosomes: a study of {Y} chromosome
  microsatellites.
\newblock {\em Molecular Biology and Evolution\/}~{\em 16}(12), 1791--1798.

\bibitem[\protect\citeauthoryear{Ratmann, Andrieu, Wiujf, and
  Richardson}{Ratmann et~al.}{2009}]{ratmann:andrieu:wiujf:richardson:2009}
Ratmann, O., C.~Andrieu, C.~Wiujf, and S.~Richardson (2009)
\newblock Model criticism based on likelihood-free inference, with an
  application to protein network evolution.
\newblock {\em Proceedings of the National
  Academy of Sciences of the United States of America\/}~{\em 106}, 1--6.

\bibitem[\protect\citeauthoryear{Ratmann, Andrieu, Wiuf, and
  Richardson}{Ratmann et~al.}{2010}]{ratmann:andrieu:wiujf:richardson:2010}
Ratmann, O., C.~Andrieu, C.~Wiujf, and S.~Richardson (2010)
\newblock Reply to Robert et al.: ``Model criticism informs model choice and model comparison". 
\newblock {\em Proceedings of the National
  Academy of Sciences of the United States of America\/}~{\em 107\/}(3), E6

\bibitem[\protect\citeauthoryear{Robert, Cornuet, Marin, and Pillai}{Robert
  et~al.}{2011}]{robert:cornuet:marin:pillai:2011}
Robert, C., J.-M. Cornuet, J.-M. Marin, and N.~Pillai (2011).
\newblock Lack of confidence in approximate {B}ayesian computation model choice.
\newblock {\em Proceedings of the National
  Academy of Sciences of the United States of America\/}~{\em 108\/}(37), 15112--15117.

\bibitem[\protect\citeauthoryear{Rousseau}{Rousseau}{2007}]{rousseau:07}
Rousseau, J. (2007).
\newblock Approximating interval hypotheses: p-values and {B}ayes factors.
\newblock In J.~M. Bernardo, M.~Bayarri, J.~O. Berger, A.~P. Dawid,
  D.~Heckerman, A.~F.~M. Smith, and M.~West (Eds.), {\em Bayesian Statistics
  8}. Oxford: Oxford University Press.

\bibitem[\protect\citeauthoryear{Rousseau and Mengersen}{Rousseau and
  Mengersen}{2011}]{mengersen:rousseau:2011}
Rousseau, J. and K.~Mengersen (2011).
\newblock Asymptotic behaviour of the posterior distribution in overfitted
  mixture models.
\newblock {\em Journal of the Royal Statistical Society, Series B\/}~{\em 73\/}(5), 689--710.

\bibitem[\protect\citeauthoryear{Sidak, Hajek and Sen}{Sidak et~al.}{1999}]{sidak:hajek:sen:1999}
Sidak, J., Hajek, Z. and Sen, P.K. (1999).
\newblock {\em Theory of Rank Tests, Second Edition}.
\newblock Academic Press.

\bibitem[\protect\citeauthoryear{Slatkin}{Slatkin}{1995}]{Slatkin:1995}
Slatkin, M. (1995).
\newblock A measure of population subdivision based on microsatellite allele frequencies.
\newblock {\em Genetics\/}~{\em 139\/}(1), 457--462.

\bibitem[\protect\citeauthoryear{Spiegelhalter, Best, Carlin, and {van der
  Linde}}{Spiegelhalter et~al.}{2002}]{spiegbestcarl}
Spiegelhalter, D.~J., N.~G. Best, B.~P. Carlin, and A.~{van der Linde} (2002).
\newblock Bayesian measures of model complexity and fit (with discussion).
\newblock {\em Journal of the Royal Statistical Society, Series B\/}~{\em 64\/}(2), 583--639.

\bibitem[\protect\citeauthoryear{Tavar{\'e}, Balding, Griffith, and
  Donnelly}{Tavar{\'e} et~al.}{1997}]{tavare:balding:griffith:donnelly:1997}
Tavar{\'e}, S., D.~Balding, R.~Griffith, and P.~Donnelly (1997).
\newblock Inferring coalescence times from {DNA} sequence data.
\newblock {\em Genetics\/}~{\em 145}(2), 505--518.

\bibitem[\protect\citeauthoryear{Toni, Welch, Strelkowa, Ipsen, and
  Stumpf}{Toni et~al.}{2009}]{toni:etal:2009}
Toni, T., D.~Welch, N.~Strelkowa, A.~Ipsen, and M.~Stumpf (2009).
\newblock {Approximate {B}ayesian computation scheme for parameter inference
  and model selection in dynamical systems}.
\newblock {\em Journal of the Royal Society Interface\/}~{\em 6\/}(31), 187--202.

\bibitem[\protect\citeauthoryear{Walsh and Raftery}{Walsh and
  Raftery}{2005}]{walsh:raftery:2005}
Walsh, D. and A.~Raftery (2005).
\newblock Classification of mixtures of spatial point processes via partial
  {B}ayes factors.
\newblock {\em Journal of Computational and Graphical Statistics\/}~{\em 14\/}(1), 139--154.


\end{thebibliography}
\end{document}